\documentclass{amsart}
\usepackage{amsmath,amssymb,amsthm,graphicx,epsfig,float,url}
\usepackage[colorlinks=true,citecolor={Plum},linkcolor={Periwinkle}]{hyperref}
\usepackage{pdfsync}
\usepackage[usenames, dvipsnames]{xcolor}
\usepackage{tikz}
\usepackage{mathtools,physics}
\usepackage{subfig}
\usepackage{stmaryrd}
\usepackage{amssymb}
\usepackage{mathrsfs}  
\usepackage{float}
\usepackage{esint}
\usepackage{verbatim}
\usepackage{geometry}
\usepackage{algorithm,algpseudocode} 
\usetikzlibrary{decorations.pathmorphing}
\usetikzlibrary{patterns.meta, plotmarks,calc}

\usepackage{dsfont}
\definecolor{cb-blue}       {RGB}{ 0, 109, 219}
\definecolor{cb-green-lime} {RGB}{ 36, 255,  36}

\newcommand{\R}{\textnormal{I\kern-0.21emR}}
\newcommand{\N}{\textnormal{I\kern-0.21emN}}
\newcommand{\Z}{\mathbb{Z}}

\def\O{{\Omega}}
\def\n{{\nabla}}
\def\p{{\varphi}}

\def\e{{\varepsilon}}

\def\fooaux#1#2{%
  \mkern2mu
  \setbox0=\hbox{\mathsurround=0pt$#1\bar{\mkern-2mu #2 \mkern2mu}$}%
  \setbox1=\hbox to \wd0{\hss$#1\cdot$\hss}%
  \vbox{\offinterlineskip\copy1\vskip-.4\ht1\box0}%
  \mkern-2mu
}

\newtheorem{theorem}{Theorem}

\newtheorem{material}{material}[section]
\newtheorem{proposition}[material]{Proposition}
\newtheorem{corollary}[material]{Corollary}
\newtheorem{definition}[material]{Definition}
\newtheorem{lemma}[material]{Lemma}

\newtheorem{remark}[material]{Remark}

\numberwithin{equation}{section}

\title{Uniqueness \& Non-uniquess for the Mean Field Control of fisheries}
\author{Greta Lamonaca}
\address{Greta Lamonaca, Universit\'e d'Orl\'eans, Universit\'e de Tours, CNRS, IDP, UMR 7013, Orl\'eans, France, CEREMADE, UMR CNRS 7534, Universit\'e Paris-Dauphine, Universit\'e PSL, Place du Mar\'echal De Lattre De Tassigny, 75775 Paris cedex 16, France} \email{greta.lamonaca@univ-orleans.fr}
\author{ Idriss Mazari-Fouquer}
\address{Idriss Mazari-Fouquer, CEREMADE, UMR CNRS 7534, Universit\'e Paris-Dauphine, Universit\'e PSL, Place du Mar\'echal De Lattre De Tassigny, 75775 Paris cedex 16, France}
\email{mazari@ceremade.dauphine.fr}
\author{Gr\'egoire Nadin}
\address{Gr\'egoire Nadin, Universit\'e d'Orl\'eans, Universit\'e de Tours, CNRS, IDP, UMR 7013, Orl\'eans, France}
\email{gregoire.nadin@cnrs.fr}
\begin{document}

\begin{abstract}
We study a Mean Field Control system arising in the management of fisheries with a special emphasis on non-uniqueness issues. Namely, we focus on a situation where a group of players coordinate in order to harvest a fishery in the most efficient way possible. A major challenge in such modelling is the coupling between the dynamics of fish population, which we model through a reaction-diffusion equation, and that of the players, which is seen through the lens of Mean Field Control. The resulting evolution system consists of four coupled equations. A central issue, both in the analysis and from the modelling perspective, is the uniqueness of solutions of this system. By focusing on the ergodic (or static) counterpart of the evolution equation, we show that one should in general expect the emergence of multiple solutions. Our approach relies on the theory of bifurcation, and the bifurcation parameter we take is the (biologically relevant) total amount of food available to the population. We also give refined uniqueness criteria that allow to bypass several limitations of previous works on this type of system \cite{zbMATH07930639}.  This fits within two growing research lines: one on the optimal harvesting of fisheries \cite{zbMATH08109759,zbMATH07930639}, one on questions of non-uniqueness in Mean Field Games and Mean Field Control \cite{zbMATH07194583,Cirant2017,zbMATH06898670}.
\end{abstract}

\maketitle


\textbf{AMS-Subject Classification:} 35K10, 35K57, 35Q89, 35Q92, 49N80, 91A12, 91A16.

\textbf{Acknowledgement:}  This project has received funding from the European Union's Horizon Europe research and innovation programme under the Marie Sk\l odowska-Curie grant agreement No 101126554. I. M.-F. is supported by the PSL Young Researcher Starting Grant (P.I: I. M.-F.) ``Optimal control of ecological systems" and G. N., G. L. are partially supported by the ANR project ReaCh ``R\'eaction-diffusion: nouveaux Challenges". I. M.-F. and G. N. are partially supported by the ANR project STOIQUES.

\textbf{Disclaimer:}  Views and opinions expressed are however those of the author(s) only and do not necessarily reflect those of the European Union. Neither the European Union nor the granting authority can be held responsible for them.

\section{Introduction}
\subsection{Scope of the paper}\label{SubSec:Scope}
Before giving the precise model we study, let us make some comments on the scientific context of the paper.

\textbf{Motivation and main goals}
The goal of this paper is to analyse some surprising phenomena arising in the optimal management of fisheries. More specifically, we are interested in a situation where a central planner is trying to harvest a fish population in order to maximise the global yield. To do so, this planner tells agents which strategy they should follow. This question stems from considerations in conservation ecology \cite{BBC50,BBCwaste,BBC2, BBC} and is a possible formulation of the question ``What is the impact of fishing on aquatic ecosystems?". From the mathematical perspective, this fits within a growing literature on the optimal management of fisheries \cite{Braverman2009,bressan2019competitive,Bressan_2013,zbMATH07006536,zbMATH07248633} and more particularly on the use of Mean Field Games (MFG) models to understand these situations \cite{zbMATH07930639,zbMATH08109759}. From a modelling perspective, let us make the following comments: first, we will be assuming that an infinite number of agents, all following the instructions of a central planner, are engaged in the harvesting of a fish population that evolves according to a monostable logistic-diffusive equation. This leads to the study of a system of four coupled equations, the Mean Field Control (MFC) system, the uniqueness of solutions of which is the central focus of this paper. Second, we will be focusing on a \emph{stationary} version of this system, traditionally referred to as the \emph{ergodic system} (and which should describe the long-time  behaviour of the underlying optimisation problem). The main results of this paper have to do with the \emph{(non-)uniqueness} of solutions of this system. We emphasise that, as we shall discuss in Section \ref{Se:Bibliography}, this question is crucial from both mathematical and modelling perspectives on Mean Field systems (either Mean Field Games or Mean Field Control). Specifically, proving the existence of multiple solutions of MFG or MFC systems has attracted much attention in recent years  \cite{zbMATH07194583,Cirant2017,zbMATH06898670}. Here, we show that, depending on the total carrying capacity, there can exist either a unique or at least two distinct solutions of the ergodic MFC system. While such phenomena are not unknown in the MFG/MFC literature, we stress that our result shows a dramatic change in behaviour depending on biologically relevant quantities, which are used as bifurcation parameters. We refer to Section \ref{Se:Bibliography} for comparisons with the existing literature.

\textbf{Model under consideration} Although our main results are stated in the one dimensional case, some statements of independent interest hold in any dimension. We thus consider a smooth bounded domain $\O\subset \R^d$, $d\geq 1$ (when $d=1$, we take $\O=(0;\ell)$). $\O$ will be dubbed ``the lake". The fish population is modelled through a density $\theta$ and, in the absence of fishermen, is subject to a Brownian diffusion with strength $\sqrt{2d}$ ($d>0$), to a Malthusian death rate $-\theta^2$ and can access resources $k\in L^\infty(\O)$, leading to a growth rate $k\cdot\theta$. Fix a time horizon $T>0$. For a given  density of fishermen $m\in L^\infty((0;T)\times \O;\R^+)$  and some initial condition $\theta_0\geq 0,\, \not\equiv 0$, the overall evolution of the fish population is governed by the monostable equation 
\[\begin{cases}\partial_t\theta_m-d\Delta\theta_m=\theta_m\left(k-m-\theta_m\right), \ \theta_m >0 &\text{ in }(0;T)\times \O,\,
\\ \partial_\nu \theta_m=0&\text{ on }(0;T)\times \partial \O,\, 
\\ \theta_m(t=0,\cdot)=\theta_0(\cdot)&\text{ in }\O.\end{cases}\]
The fishermen are subject to some random diffusion with strength $\sqrt{2\nu}$, $\nu>0$. For a given time horizon $T>0$, a  \emph{strategy} for fishermen is a given vector field $\alpha\in L^2((0;T)\times \O)$ and, given an initial probability measure $m_0$ and some strategy $\alpha$, the fishermen density  $m_\alpha$ evolves according to the conservation law
\[\begin{cases}\partial_tm_\alpha-\nu \Delta m_\alpha +\n\cdot(\alpha m_\alpha)=0, \ \fint_\O m_\alpha =1 &\text{ in }(0;T)\times \O,\, 
\\ \partial_\nu m_\alpha=0&\text{ on }(0;T)\times \partial \O,\, 
\\ m_\alpha(t=0,\cdot)=m_0&\text{ in }\O.\end{cases}\]The strategy $\alpha$ is chosen so as to solve the optimisation problem
\begin{equation}\label{Eq:CommonPv} \sup_{\alpha \in L^2((0;T)\times \O)}\iint_{(0;T)\times \O}\left(\theta_{m_\alpha}m_\alpha-\frac{\vert \alpha\vert^2}2m_\alpha\right).\end{equation}
Note that we choose a quadratic cost of control $|\alpha|^2$; this simplifies several points of the proof. As our goal is to showcase examples of non-uniqueness stemming from the reaction-diffusion equation component of the model, we feel this also allows for more readability. 

The existence of a solution of \eqref{Eq:CommonPv} is still open; nevertheless, assuming that an optimiser $\alpha^*$ exists, one can characterise it using the solution $(u,\eta,m,\theta)$ of the following coupled system (endowed with Neumann boundary conditions on each equation):
\begin{equation}\label{Eq:MFC}
\begin{cases}
-\partial_t u-\nu\Delta u-\frac{|\n u|^2}{2}=\theta(1-\eta)&\text{ in }(0;T)\times \O,
\\ -\partial_t\eta-d\Delta \eta-\eta(k-m-2\theta)=m&\text{ in }(0;T)\times \O,
\\ \partial_tm-\nu\Delta m+\n\cdot(m\n u)=0, \ \fint_\O m =1 &\text{ in }(0;T)\times \O,
\\ \partial_t\theta-d\Delta \theta-\theta(k-m-\theta)=0, \ \theta >0 &\text{ in }(0;T)\times \O,
\\ \theta(t=0)=\theta_0,\, m(t=0)=m_0,\, u(T,\cdot)=0,\, \eta(T,\cdot)=0&\text{ in }\O.
\end{cases}
\end{equation}
Indeed, it is then possible to show that $\alpha^*=\nabla u$. Under the assumption of existence of an optimal strategy, the derivation of \eqref{Eq:MFC} follows from standard techniques, and we detail it in Appendix \ref{Ap:Derivation}. An important note is that the function $u$ is \emph{not} the value function of the optimal control problem (contrary to what happens in the competitive MFG), and the Hamilton-Jacobi equation it solves is not derived using the Carath\'eodory-Bellman dynamic programming principle. Likewise, contrary to the standard MFG/MFC setting, \eqref{Eq:MFC} involves a fourth equation on an auxiliary function $\eta$, coming from the linearisation of the reaction-diffusion equation. 

The two basic questions concerning \eqref{Eq:MFC} are existence and uniqueness of solutions. The existence of solutions is, in this diffusive setting (\emph{i.e.} with $\nu>0$) always guaranteed by a standard fixed point argument which we do not detail. The uniqueness, on the other hand, is a much trickier endeavour, and the focus of our analysis. To simplify matters further, we consider the ergodic version of \eqref{Eq:MFC}: at a formal level, we expect that, as $T\to +\infty$, there should hold (see Remark \ref{Re:Ergodic})
\[ u_T(t,x)\sim \overline \lambda+\bar u(x),\, m_T\sim \bar m,\, \eta_T\sim\bar\eta,\, \theta_T\sim\bar\theta,\] where $(\bar \lambda,\bar u,\bar \eta,\bar m,\bar \theta)$ solves the so-called ``ergodic system"
\begin{equation}\label{Eq:EMFC}
    \begin{cases}
        \bar\lambda - \nu \Delta \bar u -\frac{{|\nabla \bar u|}^2}{2}= \bar \theta (1-\bar \eta) \quad &\text{in } \ \Omega,\\
         -d\Delta \bar \eta -\bar \eta (k - \bar m -2 \bar \theta)=\bar m,  & \text{in } \ \Omega, \\
        -\nu \Delta \bar m +\nabla \cdot (\bar m \, \nabla \bar u)=0, & \text{in } \ \Omega,\\
        -d\Delta \bar \theta -\bar \theta (k(x) -\bar m -\bar \theta)=0, \, \bar \theta>0 & \text{in } \ \Omega,\\
        \fint_\Omega \bar u =0, \quad \fint_\Omega \bar m =1.
    \end{cases}
\end{equation}
Our paper presents results about the (non-)uniqueness of solutions of this ergodic system. These results rely on bifurcation techniques. Roughly speaking, we state that if $\O$ is one-dimensional and if $k$ is constant (and can be thought of as the total amount of resources available to the population) then (Theorem  \ref{Th:NonUniqueness}) if $k$ is small enough, \eqref{Eq:EMFC} has multiple solutions whereas (Theorem \ref{Th:Uniqueness}), if $k$ is large enough, \eqref{Eq:EMFC} has a unique solution. 
\begin{remark}\label{Re:Ergodic}Some comments are in order:
\begin{enumerate}
\item Modelling-wise, there is a gap in our analysis: proving rigorously that $(u_T,\eta_T,m_T,\theta_T)$ stabilises as $T\to +\infty$ can only be done using arguments that would give uniqueness for the evolution system \eqref{Eq:MFC} \cite{CardaliaguetKAM,CardaliaguetPorretta,zbMATH07399451,zbMATH08109759}. 
\item As will be commented further upon, the non-uniqueness for solutions of \eqref{Eq:EMFC} is structurally tied to the cooperative (MFC) setting. We do not expect this non-uniqueness result to be true in the competitive (MFG) setting, see Section \ref{Se:CommentsNU}.\end{enumerate}
\end{remark}

\subsection{Main model and main results}\label{Se:Main}
It will be convenient to use the Hopf-Cole transform to get rid of one of the equations in \eqref{Eq:MFC}. Namely, introduce
\[ \bar\p:=e^{\frac{\bar u}{2\nu}},\] so that 
\[-2\nu^2\Delta \bar\p-\bar\theta(1-\bar\eta)\bar\p=-\bar\lambda\bar \p.\] Likewise, solving the equation on $\bar m$ explicitly, we obtain $\bar m=\bar\p^2$, and \eqref{Eq:MFC} reduces to the system of three equations
\begin{equation}\label{EMFC}\tag{S}
    \begin{cases}
        -2\nu^2 \Delta \bar\p - \bar\theta (1-\bar \eta) \bar\p = -\bar \lambda \bar\p, &\text{in } \ \O,\, 
        \\ \ \fint_\Omega \bar\p^2 =1, \ \bar\p >0 &\text{in } \ \Omega,\\
        -d\Delta \bar \theta -\bar \theta (k -\bar\p^2 -\bar \theta)=0, \ \bar\theta>0 &\text{in } \ \Omega,\\
        -d\Delta \bar \eta -\bar \eta (k-\bar\p^2 - 2\bar \theta)=\bar\p^2 &\text{in } \ \Omega,
        \\ \partial_\nu\bar \p=\partial_\nu\bar\theta=\partial_\nu\bar\eta=0 &\text{on }\partial \O.
    \end{cases}
\end{equation}
Henceforth, we make the assumption that 
\[ \fint_\O k>\fint_\O \bar\p^2=\fint_\O \bar m=1;\]this condition guarantees \cite{BHR} that for any $m$ such that $\fint_\O m=1$ there exists a unique positive solution $\theta$ of 
\[-d\Delta \theta-\theta(k-m-\theta)=0,\] and can be interpreted as saying there are more resources available than fishermen. 

\subsubsection{A uniqueness result} The core of the paper, as we mentioned, is the non-uniqueness of solutions of \eqref{EMFC}; in order to highlight that one can not expect non-uniqueness in general, we begin with a uniqueness result:
 \begin{theorem}[Uniqueness of solution of \eqref{EMFC}]\label{Th:Uniqueness}
     For any $d, \nu,\ell>0$ there exists $k_0>1$ such that for any $k\ge k_0$ the unique solution of \eqref{EMFC}, set in $\O=(0;\ell)$, is \eqref{constant_solution}.
    \end{theorem}
    The conditions that $k,\ell,\, d,\, \nu$ must satisfy are derived in Section \ref{Sec:ProofTh.Uniqueness}, see Proposition \ref{Prop:Intermediate}. We note that a key quantity is the sign of  $1-\bar\eta$, which can be interpreted as a measure of ecosystem balance or imbalance (we refer to Section \ref{Se:CommentsNU}). This uniqueness result fits in the existing literature devoted to fishing MFG  \cite{zbMATH08109759}: the strategy consists in finding regimes where the Lasry-Lions monotonicity condition (see \eqref{monotonicitynonlocalinspace} below) holds. Our proof, however, relies on fine estimates of solutions of reaction-diffusion equations, and thus allows to bypass the perturbative setting of \cite{zbMATH08109759}. This comes at the cost of a dimensional restriction. Namely, such estimates can be derived in one-dimension following the work of Bai, He \& Li \cite{BHL2015}, although we hope that our result can be extended to the higher dimensional setting. Nevertheless, we also give a conditional result (see Theorem \ref{Th:generaluniqueness} below) which essentially states that, if adequate estimates are satisfied by the solutions, then these solutions are unique.

\subsubsection{Emergence of multiple solutions in one dimension and comments}
Recall that $k$ is constant, $k>1$. Observe that, in this case, a particular solution of \eqref{EMFC} is given by 
\begin{equation}\label{constant_solution}
        \left(\bar\theta,\bar\eta, \bar \lambda,\bar\p\right)=\left(k-1, \frac{1}{k-1}, k-2, 1\right).
    \end{equation}
    This solution is referred to as the ``constant" or ``trivial" solution.
We can now state our main non-uniqueness result:
    \begin{theorem}[Non-uniqueness of solutions of \eqref{EMFC}]\label{Th:NonUniqueness} 
    For any $d, \nu >0$, there exists $k^*\in (1;2)$ such that, for any $k_1\in (1;k^*)$,  there exists $\ell(k_1)>0$ such that for any $\epsilon >0$ there exists $k\in \left( k_1-\epsilon; k_1 +\epsilon\right)$ for which system \eqref{EMFC} admits a $\ell(k_1)$-periodic non-constant solution.
    \end{theorem}
    The proof provides semi-explicit expressions for $k^*$ and $\ell(k)$, and our analysis is completed by numerical simulations.
    
    Observe that a corollary of Theorems \ref{Th:Uniqueness}--\ref{Th:NonUniqueness} is the following:
    \begin{corollary}
    For any $d,\nu>0$, there exist $\ell(k_1)>0$ and $k_0,\, k_1>1$ such that, for $k=k_1$, \eqref{EMFC} has at least two solutions, while for any $k\geq k_0$ \eqref{EMFC} has a unique solution.
    \end{corollary}

\subsubsection{Some comments and heuristics on Theorem \ref{Th:NonUniqueness}}\label{Se:CommentsNU}
Theorem \ref{Th:NonUniqueness} calls for several comments: first of all, this result will be proved using bifurcation theory, as should be expected. This is not the first time bifurcation is used to derive non-uniqueness results in MFG systems, and we refer for instance to \cite{Cirant2017}. However, the main contribution of this result is that the bifurcation parameter is the carrying capacity $k$ of the environment. Another point that needs to be made regarding the structure of the bifurcation: although we shall see (see Proposition \ref{Prop:characterizationofbifurcationbranch}) that the bifurcation is not transcritical, in a sense made precise in the proof,  we can not fully characterise it, although the numerical simulations presented in Section \ref{Se:Numerics} suggest a pitchfork bifurcation.

Let us now comment the regime in which these non-unique solutions emerge: (non-)uniqueness is intimately tied to the monotonicity condition of Lasry \& Lions \cite{LL1,LL2}, which can roughly be interpreted as ``aversion to crowds": it is not in the central planner's best interest to send agents where there is already an overabundance of agents. Under this monotonicity condition, which we detail in a subsequent paragraph of the paper, the uniqueness of solutions of \eqref{EMFC} is straightforward to derive. Here, we are typically in a regime where this monotonicity breaks down, and the fact that $k^*\in (1;2)$ indicates that we are working in regimes where resources are scarce.

This allows us to present the following heuristics for Theorem \ref{Th:NonUniqueness}: first, as was mentioned, the quantity $1-\bar \eta$ can be interpreted as a balance rate for the ecosystem, meaning that the zone $\{1-\bar\eta>0\}$ corresponds to a zone where a balance is struck between the resources available and the fishing pressure, while the zone $\{1-\bar\eta<0\}$ corresponds to zones where the fishing pressure outweighs the resources available. Now, let us give a formal intuition of the non-uniqueness result. As we saw, \eqref{constant_solution} provides a solution of \eqref{EMFC}, but we claim that this solution is likely not the optimal solution of the asymptotic version of problem \eqref{Eq:CommonPv} in certain parameter ranges (thereby leading to the non-uniqueness of solutions of \eqref{EMFC}). To explain why, we replace \eqref{Eq:CommonPv} with 
\[ \sup_{\alpha \in L^2((0;T)\times \O)}\iint_{(0;T)\times \O}\left(\theta_{m_\alpha}m_\alpha-\e\frac{\vert \alpha\vert^2}2m_\alpha\right) \] and we let $d,\, \nu,\, \e \to 0$, $T\to +\infty$. It is standard that, for a fixed $(k,\bar m)$, the solution of $-d\Delta \bar\theta-\bar\theta(k-\bar m-\bar\theta)=0$, endowed with Neumann boundary conditions, can be approximated by $(k-\bar m)_+$ as $d\to 0^+$. Recall that $\bar\alpha = \nabla \bar u$, $\bar m=\bar\p^2$ and that $\bar u=\ln(\bar \p)$. Overall, the optimal control problem under consideration should thus write (although this is probably challenging to prove)
\[ \max_{\bar\p,\, \fint_\O \bar\p^2=1}\fint_\O \bar\p^2(k-\bar\p^2)_+-\e\fint_\O |\n \bar\p|^2.\] Taking the limit $\e\to 0^+$, we are left with studying 
  \[ \bar J (\bar\p^2) := \fint_\Omega \left(k-\bar\p^2\right)_+ \ \bar\p^2\]
    and we aim at proving that $\bar\p^2\equiv 1$ (corresponding to a uniform distribution of fishermen) is not the solution of     \[ \max_{\tiny{\fint_\Omega \bar\p^2 =1}} \bar J (\bar\p^2).\]
Note that for $\bar\p^2\equiv 1$ we have
    \[ \bar J(1) =(k-1) \le \frac{k^2}{4}.\]
    We now assume $k<2$ and we split the lake $\Omega$ into a balanced zone and an unbalanced zone $U$. Namely, fix  $M>k>1$ and consider the fishermen distribution
    \[ \bar\p^2_U =
    \begin{cases}
      \frac{k}{2}   \quad &\text{in } \ \Omega \setminus U, \\
      M  &\text{in } \ U.
    \end{cases}\]
    The zone $\Omega\setminus U$ is balanced, whereas $U$ is not.  Further observe that $(M,|U|,k,|\O|)$ are linked by the normalisation condition $\fint_\O \bar\p^2_U=1 $, which gives 
    \[ |U|= |\O|\frac{ 1-\frac{k}{2}}{M - \frac{k}{2}}.\]
    The gain associated with $\bar\p_U$ is explicit, and given by    \begin{align*}
        \bar J(\bar\p^2_U) = \frac{k^2}{4}\cdot\frac{M-1}{M-\frac{k}{2}}.
    \end{align*}
    We want to find regimes where $\bar J(\bar\p^2_U)\ge \bar J(1)$ \emph{i.e.} such that
    \[ \frac{k^2}{4}\frac{M-1}{M-\frac{k}{2}} - k +1 \ge 0.\]
By standard computations, it suffices to take    \[M \ge \frac{k}{2-k}\] and so, if $M$ is large enough, a heterogeneous distribution of fishermen gives a larger common income than a homogeneous one. We emphasise once more that this reasoning is formal, and making it rigorous seems challenging. From an applied perspective, this result highlights that in a coordinated unfavourable framework ``sacrificing'' a portion of the domain to over-exploitation (namely, creating spatial heterogeneity) can lead to a higher collective income, a phenomenon that would not arise under purely competitive (selfish) behaviour: indeed, in the competitive scenario, a single player who over-exploits resources in one area would likely move to a more favorable region, ultimately causing the extinction of the whole fish population (purely selfish behaviour), and consequently derive a zero individual gain as $T\to +\infty$. \\

Finally, let us conclude with another remark regarding the specificity of the coordinated setting.
The emergence of spatially heterogeneous harvesting patterns bears some resemblance to the exclusion zone mechanisms studied in predator-prey reaction-diffusion models in \cite{berestycki2026}, where limiting access to portions of the domain can significantly influence species persistence and coexistence. However, the underlying mechanisms are fundamentally different: in these models, predators follow prescribed biological dynamics and are not modeled as rational agents, while in our framework the spatial distribution of harvesting effort arises endogenously from the optimization problem solved by the agents.

This distinction is also evident with respect to harvesting models based on fixed strategies, such as the multispecies fishery model of \cite{Villain2025}. In contrast, the Mean Field Game formulation allows agents to adapt their behavior to both the spatial distribution of the resource and the strategic behavior of agents through an optimization process. This allows for the joint analysis of ecological dynamics and strategic decisions of agents in harvesting problems.\\

We further note that, whenever $k$ is constant and $\O=(0;\ell)$, we will show in an upcoming paper \cite{LMFN27} that there exists a unique solution of the ergodic Mean Field Game.

\subsubsection{Numerical illustrations of the bifurcation}\label{Se:Numerics}
We present several numerical simulations illustrating the oscillatory solutions constructed through the bifurcation analysis. Throughout this section, we consider the case $d=2\nu^2$.\\

\textbf{One-dimensional simulations.} We first focus on the one-dimensional setting studied in Sections \ref{Sec:ProofTh.Bifurcation}--\ref{Sec:ProofTh.Uniqueness}. Using symbolic computations implemented with the \emph{sympy} library, we compute the critical threshold value $k^*$ associated with the bifurcation. This identifies the interval $k\in (1;k^*)$ where nontrivial oscillatory solutions are expected to emerge from the constant state. To numerically solve \eqref{EMFC}, we adopt a variational (energy-based) approach. More precisely, we use the \emph{eig} function from the \emph{scipy} library to solve the constrained linear eigenvalue problem associated with the $\bar\p$-equation, while the variables $\bar\theta$ and $\bar\eta$ are updated through a gradient descent scheme with adaptive stepsizes:
    \[ \bar \theta_{n+1} = \bar \theta_n -t_{E_1,n}dE_1(\bar \theta_n), \quad \bar\eta_{n+1} = \bar \eta_n -t_{E_2,n} dE_2(\bar\eta_n).\]
    
    \captionsetup[algorithm]{justification=raggedright, singlelinecheck=false}
    \begin{algorithm}
    \caption{Numerical resolution of \eqref{EMFC} for $k\in(1,k^*)$}
    \label{alg.label}
    \begin{algorithmic}[1]
    \Require $\bar\theta_0$, $\bar\eta_0$, $\bar\p_0$, tol, num\_iterations, $k$
    \State Initialisation $\bar\theta$ $\gets$ $\bar\theta_0$, $\bar\eta$ $\gets$ $\bar\eta_0$, $\bar\p$ $\gets$ $\bar\p_0$, $t_{E_1,0} \gets 0.01$, $t_{E_2,0} \gets 0.00001$
    \State $n\gets 0$
    \While{$\|dE_1(\bar\theta, \bar\p, k)\| > $ tol \textbf{and} $\|dE_2(\bar\eta, \bar\theta, \bar\p, k)\| > $ tol \textbf{and} $n<$ num\_iterations}
        \State Compute $\bar\lambda_{n+1}$, $\bar\p_{n+1}$
        \State Compute $\bar\theta_{n+1}$ with a gradient descent
        \State Compute the new stepsize $t_{E_1}$
        \State Compute $\bar\eta_{n+1}$ with gradient descent
        \State Compute the new stepsize $t_{E_2}$
        \State Update $\bar\theta$, $\bar\eta$, $\bar\p$, $t_{E_1}$, $t_{E_2}$
        \State $n \gets n+1$
    \EndWhile
    \end{algorithmic}
    \end{algorithm}
    We do not address the convergence analysis of the above algorithm. Instead, we rely on the stopping criterion
    \[\|dE_1(\bar\theta,\bar\p,k)\| \leq \mathrm{tol}, \qquad \|dE_2(\bar\eta,\bar\theta,\bar\p,k)\| \leq \mathrm{tol},\]
    which numerically indicates that the iteration has reached a critical point of the associated energy functional.\\
    \begin{figure}[h!]
        \centering
        \includegraphics[width=0.7\linewidth]{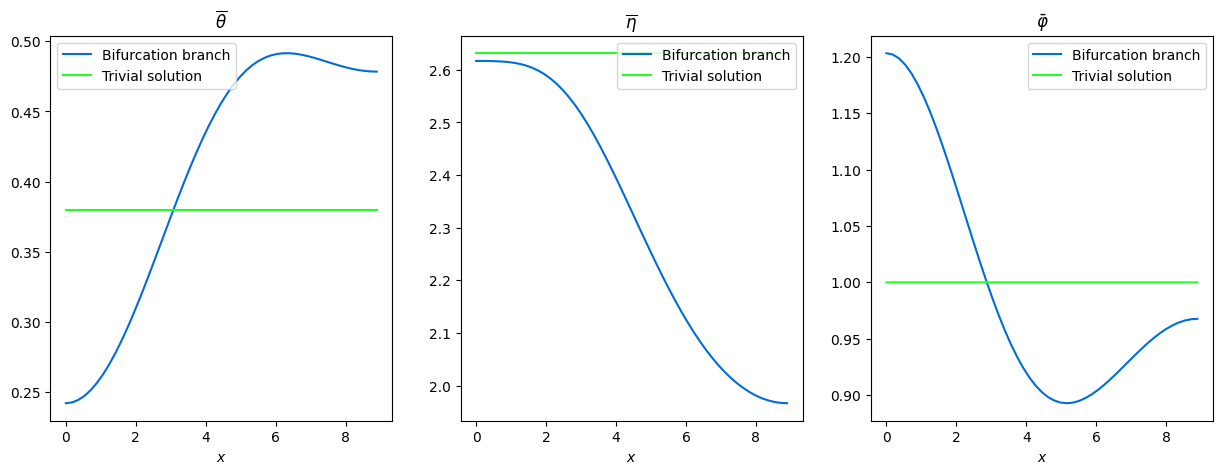}
        \caption{\centering{\textcolor{cb-blue}{Oscillating solution} and \textcolor{cb-green-lime}{Trivial solution} of \eqref{EMFC} for $l\simeq 9$, $d=2\nu^2$ and $k\equiv 1.38$.}}
        \label{fig.bifurcation_X_3}
    \end{figure}
    
    As should be expected, the bifurcation is not transcritical (see Proposition \ref{Prop:characterizationofbifurcationbranch}). Moreover, our numerical simulations (see Fig. \ref{fig:characterization_bifurcation}) suggest a pitchfork shape. However, proving the supercritical nature of this bifurcation remains a challenging open problem.
        \begin{figure}[h!]
            \centering
            \includegraphics[width=\linewidth]{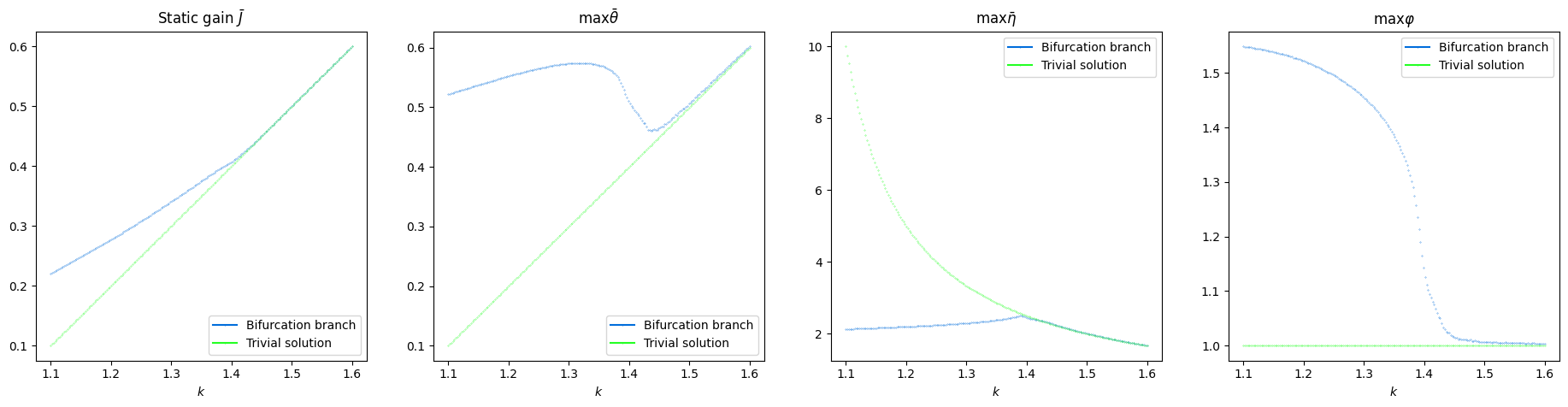}
            \caption{Numerical characterization of the bifurcation branch.}
            \label{fig:characterization_bifurcation}
        \end{figure}
    
    Interestingly, the numerical simulations (Fig. \ref{fig:characterization_bifurcation}) also suggest that the bifurcation is global in $(1;2)$, as the bifurcation branch appears to persist well beyond a neighbourhood of the critical threshold $k^*$. A formal analytical treatment of the global bifurcation is left as an open problem. In particular, while uniqueness of the trivial solution is known for sufficiently large values of $k$ (see Section \ref{Sec:ProofTh.Uniqueness}), the analytical behaviour of solutions in the intermediate regime remains open.\\
    
    \textbf{Two-dimensional simulations and Turing patterns.} Although the theoretical analysis of this work is primarily one-dimensional, we additionally present several two-dimensional simulations in order to illustrate the richer phenomenology of the system beyond the bifurcation regime covered by our analysis. Contrary to the one-dimensional setting, we do not attempt to analytically determine the critical threshold. Instead, in Fig. \ref{fig:2D_bifurcation_patterns} we vary the parameter $k$ between the values $1.2$ and $1.9$, which lie beyond the one-dimensional bifurcation threshold $k^*$ for $d=2\nu^2$, and we investigate the qualitative structure of stationary states for a fixed domain $\Omega := \left(0; 4\pi\right)^2$.\\
    \begin{figure}[h!]
        \centering
        \includegraphics[width=\linewidth]{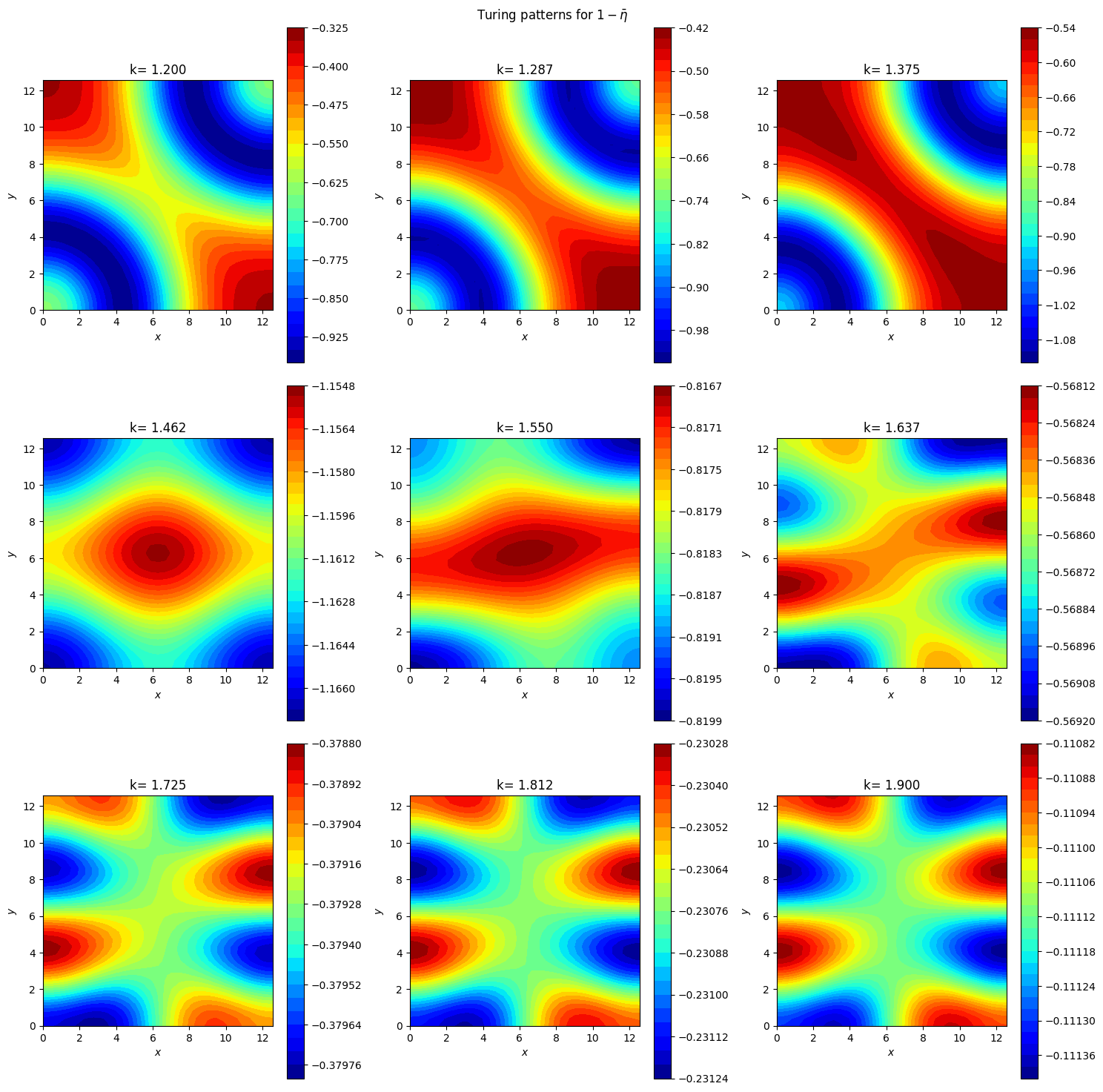}
        \caption{Turing patterns for $1-\bar\eta$ for $k \in (1.2;1.9)$ in $\left(0; 4\pi\right)^2$.}
        \label{fig:2D_bifurcation_patterns}
    \end{figure}    
        
    Our simulations reveal the emergence of a variety of Turing-type patterns, including stripes, spots and mixed configurations. Interestingly, the observed stationary states strongly depend on both the geometry of the domain and the choice of the initial perturbation. In particular, different initial datum $\bar\p_0$ lead to significantly different asymptotic configurations. We also observe that the amplitude of the patterns is really small when $k$ is approaching $2$.\\

    These numerical observations suggest that the multidimensional problem possesses a significantly richer bifurcation structure than the one-dimensional theory captures. For this reason, we believe that including these simulations provides valuable qualitative insight into the landscape of stationary solutions associated with \eqref{EMFC}, even though a rigorous analytical treatment in dimension two lies beyond the scope of the present work.

    \subsection{Structure of the paper}
    Although we presented uniqueness results before the theorem devoted to the emergence of multiple solutions as it seems more natural from a presentation perspective, we will first prove the emergence of multiple solutions in Section \ref{Sec:ProofTh.Bifurcation}.  The proof of our uniqueness result is carried out in Section \ref{Sec:ProofTh.Uniqueness}. Both sections are mathematically independent.
    
\subsection{Bibliographical references}\label{Se:Bibliography}

\subsubsection{Reaction Diffusion Equations}We review some basic facts about the dynamics of fishes in the absence of fishermen, a general reference is \cite{zbMATH07668634}. The fish population is modelled through the standard logistic diffusive equation
\begin{equation}\label{RDE_without_m}\tag{RDE}
    \begin{cases}
        \partial_t \theta - d \Delta \theta = \theta (g(x) -\theta)  &\text{in } \R_+ \times \Omega,\\
        \partial_\nu \theta =0 &\text{on } \R_+ \times \partial\Omega,\\
        \theta(0,\cdot)=\theta_0(\cdot)\not\equiv0,\ge 0 &\text{in } \Omega.
    \end{cases}
\end{equation}
 Due to its ability to capture relevant biological phenomena, this equation has been a cornerstone of mathematical biology (\cite{S,SK97} since the pioneering papers of \cite{S, F37} and \cite{K37}).  From the analytical side (\emph{i.e.} well-posedness of the model), the main references we rely on are \cite{CC, CC2} in bounded domains and  \cite{BHR} in the periodic setting, in which the focus is on the existence of non-trivial steady states, and on their dynamical stability. As it turns out, the existence, uniqueness and stability of a non-trivial steady-state $\theta$  are strictly related to the principal eigenvalue $\lambda_1[-d\Delta -g]$ of the following problem:
\[\begin{cases}
    -d\Delta \phi - g(x)\phi = \lambda_1 \phi, \ \phi>0 \quad &\text{in } \Omega,\\
    \partial_\nu\phi = 0 &\text{on }\partial \Omega.
\end{cases}\] More specifically, the existence, uniqueness and stability of $\theta$ is equivalent to 
\[ \lambda_1[-d\Delta-g]<0.\]  When this condition is met, any initial population $\theta_0$ survives in long time, and we speak of ``species persistence".
As it follows from the Rayleigh quotient formulation of the eigenvalue that \[ -\max_{\bar\Omega} g \le \lambda_1[-d\Delta -g(x)] \le -\fint_\Omega g\] we deduce that a sufficient condition for species persistence is $\fint_\O g>0$. These results led to several papers investigating the influence of the spatial heterogeneity $g$ on several criteria, including  optimal survival ability \cite{CC, CC2, CCH, SK97, BHR}.

\subsubsection{Mean Field Game \& Mean Field Control} Mean field games were introduced simultaneously and independently by Lasry \& Lions \cite{LL0, LL1, LL2} and Caines, Huang \& Malham\'e \cite{CHM} to model the behaviour of a  large number $N$ of interacting agents. These agents are supposed to be rational and indistinguishable, interacting through the empirical averages of quantities which depend on the state variable. At the limit when $N \to +\infty$, and under suitable assumptions \cite{zbMATH06987368}, the game can be approximated through a system of coupled PDEs, the so-called MFG system. A general reference for the study of MFG is \cite{zbMATH06721976,zbMATH06721977}. A classical (second order) MFG system is
\begin{equation} \label{MFG_CLLP}\begin{cases}
    -\partial_t u -\nu \Delta u - \frac{\abs{\nabla u}^2}{2}=F[m(t, \cdot)] \quad & \text{in } (0,T)\times \Omega,\\
    \partial_t m -\nu \Delta m +\nabla \cdot (m\, \nabla u) =0, \ \fint_\Omega m=1 & \text{in } (0,T)\times \Omega,\\
    u(T, \cdot)= 0, \ m(0,\cdot)=m_0(\cdot)\geq 0, \ \fint_\Omega m_0=1 & \text{in } \Omega,\\
\end{cases}\end{equation}
with $F : L^2(\Omega) \longrightarrow\R$ a running cost. Although a heterogeneous final-time condition $u(T, \cdot) = G(\cdot, m(T))$ and a general Hamiltonian $H$ could also be considered, we consider a zero final-time condition for the value function $u$ and a quadratic Hamiltonian $H(p)=\frac{1}{2}p^2$ to maintain consistency with our framework. As discussed earlier, the system \eqref{EMFC} does not strictly fit within the classical MFG setting \eqref{MFG_CLLP}, as the running cost $\theta(1-\eta)$  is non-local both in space and time. 

Let us discuss some well-known results for \eqref{MFG_CLLP}. The existence of solutions of \eqref{MFG_CLLP}, under standard regularity assumptions on $F$ as a function of $m$ and on $F[m]$ as a function of time and space, follows from a standard fixed point argument \cite{zbMATH06147631}. For the ergodic problem \eqref{Eq:EMFC}, a similar approach applies to $\bar m \longmapsto \left(\bar \lambda, F[\bar m]\right)$, where $F[\bar m]=\bar\theta[\bar m] (1-\bar\eta[\bar m])$, with minor modifications of  \cite[Appendix]{zbMATH06147631}. Establishing (non-)uniqueness of solution of \eqref{MFG_CLLP} is the main challenge. There are three regimes for which uniqueness of solution of \eqref{MFG_CLLP}  is well-known: the so called Lasry-Lions monotonicity condition, introduced in \cite{LL2}, which reads 
\begin{equation} \label{monotonicitynonlocalinspace}
    \int_\Omega (F[m_1] - F[m_2]) (m_1-m_2) \, dx \le 0 \quad \forall m_1\neq m_2.
\end{equation}
This condition is usually interpreted as aversion to crowds. The second setting for uniqueness is the small-time horizon case. This approach is not relevant here since we investigate the ergodic problem, which corresponds to the $T\to +\infty$ asymptotics of the evolution system. This topic was first introduced in a lecture of \cite{LCollegedeFrance} on January 9th 2009, and later revisited in \cite{zbMATH07194583}, where it is shown that the regularity of the Hamiltonian $H$ plays a crucial role in proving uniqueness for small-time horizons. In the absence of such regularity, counterexamples can be constructed, as illustrated in {\cite{zbMATH07194583}}. The last case for uniqueness is the small-data one (see \cite{zbMATH07194583, BC}), which is not relevant for ergodic problems either.

Regarding the long-time behaviour for MFG systems, the first contributions were the works of Cardaliaguet, Lasry, Lions and Porretta \cite{zbMATH06147631,zbMATH06250859} (for the first-order case ($\nu=0$) we refer to \cite{CardaliaguetKAM}). Under the assumption \eqref{monotonicitynonlocalinspace}, one can show the convergence of \eqref{MFG_CLLP} to the ergodic MFG system 
\begin{equation}\label{EMFC_CLLP}
    \begin{cases}
        \bar  \lambda -\nu\Delta \bar u - \frac{\abs{\nabla \bar u}^2}{2}=F\left[\bar m(\cdot)\right], \ \fint_\Omega \bar u =0 \quad &\text{in } \ \Omega,\\
        -\nu \Delta \bar m +\nabla \cdot (\bar m \cdot \nabla \bar u)=0, \ \fint_\Omega \bar m =1, \bar m>0 &\text{in } \ \Omega.
    \end{cases}
\end{equation}
This \eqref{EMFC_CLLP} system fits with our \eqref{Eq:EMFC} system if $F[\bar m]=\bar\theta[\bar m](1-\bar\eta[\bar m])$. 

While the bulk of MFG literature focuses on mean field interactions involving only the distributions of states,  Mean Field Games of Controls (or Mean Field Control) address a different scenario where the gain of an individual agent depends on the joint distribution of both states and optimal strategies. Early foundational works by \cite{GV13, GPV13} established the existence of solutions of MFC systems, particularly for stationary games under small-parameter assumptions. Further results regarding existence and uniqueness in the MFC framework can be found in \cite{BLL19, zbMATH06898614, CD18, CardaliaguetKAM, K21, K22}. Uniqueness remains a pivotal challenge in MFG theory and the Lasry-Lions monotonicity condition for MFC systems is discussed in \cite{GV13, CD18, K22}.

\subsubsection{(Non-)Uniqueness in MFG}
The question of non-uniqueness in MFG has been investigated for the time evolution problem \eqref{monotonicitynonlocalinspace} with a strong emphasis on providing instances of non-uniqueness. A first possibility involves modifying the Hamiltonian $H$, for instance considering an $H$ lacking symmetry or regularity \cite{zbMATH07194583}; remarkably, this can lead to the existence of several optimal strategies yielding the same pay-off. Modelling wise, this approach can not be suited to our needs. Another established direction focuses on the structural properties of the cost functionals. As explored in the works of Bardi \& Cirant and Bardi \& Fisher \cite{BC,zbMATH07194583}, non-uniqueness can arise as a consequence of the failure of the Lasry-Lions monotonicity condition. A representative class of such functionals is given by:
\begin{equation} F(x, m) = a_1 x \int_\Omega f(x,y) \, m(t,y) \, dy, \end{equation}
where $x$ denotes the position of a player, $m$ represents the distribution of the other players and $f(x,y)$ describes the influence of the position of other players on the cost. Within this framework, the sign of the coefficient $a_1$ dictates the qualitative behaviour of the agents and can lead to the emergence of multiple solutions. 

While the above example construct \emph{ad hoc} functionals that provide clear mathematical insights into how uniqueness can fail, it is often disconnected from the specific  mechanism of resource harvesting. In contrast, the non-uniqueness that we obtain arises directly from the problem structure.

\subsubsection{Harvesting Mean Field Game.} The mathematical analysis of the harvesting problem within the context of Mean Field Games has seen recent interest and developments in \cite{zbMATH07930639,zbMATH08109759}. The most closely related work is  \cite{zbMATH08109759}, where the authors address a MFG closely related to \eqref{EMFC} in the competitive setting and establish the long-time behaviour in a perturbative setting, first establishing global in time uniqueness for the perturbed model. The question of non-uniqueness for general ranges of parameters is left completely open.    
    
\subsection{Open problems}
We present some open problems that seem relevant and challenging:
\begin{enumerate}
\item The first one is the identification, in the presence of multiple solutions of \eqref{EMFC}, of the optimal one. This is far from trivial, and, in fact, for the different solutions we obtain in this paper, our computations were not conclusive. 
\item The previous question is intimately tied to the question of stability of the bifurcating solutions we obtain. More generally, there are still several challenging open problems about the types of bifurcations we study; in particular, determining whether these bifurcations are global or not is both important from the applied perspective, and challenging from a mathematical one.
\end{enumerate}

\section{Non-uniqueness results: proof of Theorem \ref{Th:NonUniqueness}}\label{Sec:ProofTh.Bifurcation}
\subsection{Strategy of proof and set-up}
As explained in the introduction, Theorem \ref{Th:NonUniqueness} relies on bifurcation theory and, more specifically, on an application of the Crandall-Rabinowitz theorem which we recall below. To set the terminology, let us recall the following definition:
\begin{definition}\label{De:Bifurcation}
Let $X,\, Y$ be two Banach spaces and $\mathcal F:\R\times X\to Y$ be continuous. Assume that the equation  $\mathcal{F}(\mu, x)=0$ has a curve of solutions $\left\{ (\mu, 0) \, |\, \mu \in \R\right\}$, that will henceforth be referred to as the trivial solutions. A point $(\mu_0,0)$ such that $\mathcal{F}(\mu_0,0)=0$ is called a \emph{bifurcation point} if, in any neighbourhood of $(\mu_0,0)$, there exist  $(\mu, x)$  such that $\mathcal F(\mu,x)=0$ and $x\neq 0$. 
\end{definition}
One of the standard tools in bifurcation theory is the Crandall-Rabinowitz theorem \cite{zbMATH03348766}:
    \begin{theorem}[Crandall-Rabinowitz]\label{Th:CR}
        Let $X, \, Y$ be Banach spaces and $\mathcal{F} : \R \times X \to Y$ be  $\mathscr C^2$. Let $\mu_0\in \R$. Assume that the following are satisfied:       \begin{enumerate}
            \item $\mathcal{F}(\mu,0)=0$ for any $\mu$ in a neighbourhood of $\mu_0$. \label{(i)Th:localbifurcation}
            \item $\ker \left( d_x \mathcal{F}(\mu_0,0)\right)=\R \xi_0$ for some $\xi_0 \in X\setminus \left\{ 0 \right\}$,\label{(ii)Th:localbifurcation}
            \item $\Im   \left(d_x \mathcal{F}(\mu_0,0)\right)$ has co-dimension $1$,\label{(iii)Th:localbifurcation}
            \item \emph{Transversality condition:} $d_x d_\mu \mathcal{F}(\mu_0,0)[\xi_0] \not\in \Im \left(  d_x \mathcal{F}(\mu_0,0)\right)$. \label{(iv)Th:localbifurcation}
        \end{enumerate}
        then $(\mu_0, 0) \in \R \times X$ is a bifurcation point. We note that conditions \eqref{(ii)Th:localbifurcation}--\eqref{(iii)Th:localbifurcation} imply that $d_x\mathcal F(\mu_0,0)$ is an index 0 Fredholm operator.
    \end{theorem}  
The goal is to use this theorem, which requires defining the right function $\mathcal F$. We first introduce the map
   
    \[\mathcal G: \R\times (W^{2,2}(\O))^2\ni (k,\hat \theta,\hat \eta)\mapsto \begin{pmatrix}
    -d\Delta \hat\theta-\hat\theta\left(k-\hat\p_{k,\hat\theta,\hat\eta}^2-\hat\theta\right)\\ -d\Delta\hat\eta-\hat\eta\left(k-\hat\p_{k,\hat\theta,\hat\eta}^2-2\hat\theta\right)-\hat\p_{k,\hat\theta,\hat\eta}^2
    \end{pmatrix}\] where, for any $(\hat\theta,\hat\eta)$, $\hat\p_{k,\hat\theta,\hat\eta}$ is the normalised eigenfunction associated with the principal eigenvalue $\hat \lambda$ of $-2\nu^2\Delta-\hat\theta(1-\hat\eta)$, namely, the unique solution of
    \[
    \begin{cases}
    -2\nu^2\Delta \hat\p_{k,\hat\theta,\hat\eta}-\hat\p_{k,\hat\theta,\hat\eta}\hat\theta(1-\hat\eta)=-\hat\lambda\hat\p_{k,\hat\theta,\hat\eta}&\text{ in }\O,\, 
    \\ \hat\p_{k,\hat\theta,\hat\eta}> 0&\text{ in }\O,\, 
    \\ \partial_\nu \hat\p_{k,\hat\theta,\hat\eta}=0&\text{ on }\partial \O,\, 
        \\ \fint_\O \hat\p_{k,\hat\theta,\hat\eta}^2=1.
    \end{cases}
    \]
We then define    
\begin{multline*}\mathcal F: \R\times (W^{2,2}(\O))^2\ni (k,\tilde \theta,\tilde \eta)\mapsto \mathcal G \left(k,k-1+\tilde\theta,\frac{1}{k-1}+\tilde\eta\right)\\=\begin{pmatrix}
 -d\Delta \tilde\theta-(k-1+\tilde\theta)\left(1-\tilde\p_{k,\tilde\theta,\tilde\eta}^2-\tilde\theta\right)\\ -d\Delta\tilde\eta+\left(\frac1{k-1}+\tilde\eta\right)\left(k-2+\tilde\p_{k,\hat\theta,\hat\eta}^2+2\tilde\theta\right)-\tilde\p_{k,\tilde\theta,\tilde\eta}^2
\end{pmatrix} \end{multline*}  
where $\tilde\p_{k,\tilde\theta,\tilde\eta}$ is the principal eigenfunction, defined as the unique solution of
  \[
    \begin{cases}
    -2\nu^2\Delta \tilde\p_{k,\tilde\theta,\tilde\eta}-\tilde\p_{k,\tilde\theta,\tilde\eta}(k-1+\tilde\theta)(\frac{k-2}{k-1}-\tilde\eta)=-\tilde\lambda_{k,\tilde\theta,\tilde\eta}\tilde\p_{k,\tilde\theta,\tilde\eta}&\text{ in }\O,\, 
    \\ \tilde\p_{k,\tilde\theta,\tilde\eta}\geq 0&\text{ in }\O,\, 
    \\ \partial_\nu \tilde\p_{k,\tilde\theta,\tilde\eta}=0&\text{ on }\partial \O,\, 
        \\ \fint_\O \tilde\p_{k,\tilde\theta,\tilde\eta}^2=1.
    \end{cases}
    \]
Observe that 
\begin{equation}\label{Eq:00} \tilde\p_{k,0,0}\equiv 1,\, \tilde\lambda_{k,0,0}=k-2.\end{equation}
Furthermore, the simplicity of $\tilde \lambda_{k,\tilde\theta,\tilde\eta}$ entails its $\mathscr C^\infty$ regularity as a function of $(\tilde\theta,\tilde\eta)$, thereby establishing the following lemma:
\begin{lemma}
$\mathcal F$ is $\mathscr C^2$ in a neighbourhood of $(k,0,0)$.
\end{lemma}
Finally, we note that, by construction, there holds
\[ \text{ For any } k \text{ in $(1;+\infty)$ } \ \mathcal F(k,0,0)=0.\] Our goal is now to establish that there exists $k^*\in (1;2)$ such that any $k\in (1;k^*)$ is a bifurcation point for $\mathcal F$. Using \eqref{Eq:00}, we obtain the following: for any $(\xi,\eta)\in \left(W^{2,2}(\O)\right)^2$, 
\begin{equation}\label{dF[xi,zeta]}
    d_{\tilde\theta,\tilde\eta}\mathcal F(k,0,0)[\xi,\zeta]=\begin{pmatrix}
    -d\Delta \xi+(k-1)(2\dot \p_{\xi,\zeta}+\xi)\\ 
    -d\Delta \zeta+(k-1)\zeta +\frac2{k-1}\xi-\frac{2(k-2)}{k-1}\dot\p_{\xi,\eta}
\end{pmatrix}\end{equation}
where $\dot\p_{\xi,\eta}$ is the unique solution of 
\begin{equation}\label{Eq:DotPhi}
\begin{cases}
-2\nu^2\Delta\dot\p_{\xi,\eta}+\frac{2-k}{k-1}\xi+(k-1)\zeta=-\dot\lambda_{\xi,\eta}&\text{ in }\O,\, 
\\ \fint_\O \dot\p_{\xi,\eta}=0.
\end{cases}
\end{equation}
We will proceed as follows:  in Section \ref{Se:Fredholm}, we prove that $d_{\tilde\theta,\tilde\eta}\mathcal F(k,0,0)$ has Fredholm index 0 by proving its kernel has dimension 1 and its image has co-dimension 1. Finally, in Section \ref{Se:Transversality}, we establish the transversality condition.

\subsection{Kernel and Fredholm index of $d_{\tilde\theta,\tilde\eta}\mathcal F(k,0,0)$}\label{Se:Fredholm}
Our goal in this paragraph is to check that, for some $k^*\in (1;2)$ and for any $k\in (1;k^*)$, conditions \ref{(ii)Th:localbifurcation}--\ref{(iii)Th:localbifurcation} of Theorem \ref{Th:CR} are satisfied.
\subsubsection{Basic information regarding $d_{\tilde\theta,\tilde\eta}\mathcal F(k,0,0)$}
Let $(\xi,\zeta)\in \ker(d_{\tilde\theta,\tilde\eta}\mathcal F(k,0,0))$.  As
\[\begin{cases}
-d\Delta \xi+(k-1)(2\dot \p_{\xi,\zeta}+\xi)=0 &\text{ in }\O,\,\\ 
-d\Delta \zeta+(k-1)\zeta +\frac2{k-1}\xi-\frac{2(k-2)}{k-1}\dot\p_{\xi,\eta}=0 &\text{ in }\O,\,
\end{cases}\] and as $\fint_\O \dot\p_{\xi,\eta}=0$, integrating this system in space gives
\[\fint_\O \xi=\fint_\O \zeta=0.
\]
Consequently, integrating \eqref{Eq:DotPhi} in space yields
\[\dot \lambda_{\xi,\zeta}=0.\]
Recall that $\O=(0;\ell)$ and introduce
\begin{equation}\label{Eq:omega}\omega:=\frac{2\pi}\ell.\end{equation} Decomposing $(\xi,\zeta,\dot\p_{\xi,\zeta})$ as
\[ \xi=\sum_{n\in \Z}a_ne^{in\omega x}\,, \zeta=\sum_{n\in \N}b_ne^{in\omega x},\, \dot\p_{\xi,\zeta}=\sum_{n\in \Z}c_ne^{in\omega x}\] with 
\[ a_0=b_0=c_0=0,\, a_{-n}=\overline{a_n}\] we observe that, for any $n\geq 1$,
\[ \left(A_k+n^2\omega^2D_k\right)\begin{pmatrix}a_n\\b_n\\c_n\end{pmatrix}=\begin{pmatrix}0\\0\\0\end{pmatrix},\] where
\[ A_k:=\begin{pmatrix} k-1&0&2(k-1)\\ \frac2{k-1}&k-1&2\frac{2-k}{k-1}\\ \frac{2-k}{k-1}&k-1&0\end{pmatrix},\, D_k:=\begin{pmatrix}d&0&0\\ 0&d&0\\ 0&0&2\nu^2\end{pmatrix}.\]Consequently
\[ \begin{pmatrix}a_n\\b_n\\c_n\end{pmatrix}\neq 0\Rightarrow \det(A_k+n^2\omega^2D_k)=0.\]
We introduce the polynomial 
\[ P_k:X\mapsto \det(A_k+XD_k),\] and our goal, from now on, is to establish the following: there exists $k^*\in (1;2)$ such that, for any $k\in (1;k^*)$, there exists $\ell>0$ satisfying the following: there exists a unique $n\in \N\setminus\{0\}$ such that $\det(A_k+n^2\omega^2D_k)=0$, where we recall that $\omega$ is defined in \eqref{Eq:omega}. We refer to Proposition \ref{Prop:ker_dimension1} below. 

We observe that a tedious but straightforward computation gives 
\begin{equation}\label{Eq:Pk}
P_k(X)=2\nu^2d^2 X^3+4\nu^2d(k-1)X^2 +2dX \left[\frac{\nu^2}{d}(k-1)^2 -2(2-k)\right] +4 (k-1)^2.\end{equation} As we are looking for positive roots of $P_k$, which does not depend on $\ell$, it is worth observing the following fact (which is an immediate consequence of $P_k(0)=4(k-1)^2>0$ and $P_k(x)\to-\infty$ as $x\to-\infty$):
\begin{lemma}\label{Le:PositiveRoots}
For any $k>1$, $P_k$ has at most two positive roots.
\end{lemma} 

Our strategy reads as follows:
\begin{enumerate}
\item We first show that there exists $k^*\in (1;2)$ such that, for any $k\in (1;k^*)$, $P_k$ has exactly two distinct positive roots $0<X_2<X_3$. This is Lemma \ref{Lem:properties_of_DeltaP(k)}. 
\item We then define $\ell$ as $\frac{2\pi}\ell=\sqrt{X_3}$, so that $n=1$ is the only integer solution of $\det(A_k+n^2\omega^2D_k)=0$, thereby guaranteeing that $\dim(\ker(d_{\tilde\theta,\tilde\eta}\mathcal F(k,0,0)))=1$ for any $k\in (1;k^*)$. 
\item The fact that $d_{\tilde\theta,\tilde\eta}\mathcal F(k,0,0)$ has Fredholm index 0 is then a consequence of elementary linear algebra; we refer to Proposition \ref{Prop:ker_dimension1}.
\end{enumerate}

     \begin{lemma}\label{Lem:properties_of_DeltaP(k)}
        For any $d,\nu>0$, there exists $k^*>1$ such that $P_{k}$ has a double positive root $X^*$ and such that, for any $k \in (1,k^*)$, $P_{k^*}$ has exactly two distinct positive real roots $0<X_2<X_3$. Moreover, 
        \[1< k^*<2.\]
    \end{lemma}
   
            \begin{proof}[Proof of Lemma \ref{Lem:properties_of_DeltaP(k)}]
First, observe that if $k=1$, $P_1$ has exactly three real roots, namely,         \[X_1=-\sqrt{\frac{2}{d\nu^2}}, \quad X_2=0, \quad X_3=\sqrt{\frac{2}{d\nu^2}}.\] Second, note that 
\[P_1'(0)<0\] and define, for any $k\geq 1$, $q_k:=\min_{\R_+}P_k$, so that $q_1<0$. Third, observe that a simple computation shows that, for any $x\geq 0$, $k\mapsto P_k(x)$ is increasing in $k$ (see Fig. \ref{fig:cubic polynomial} for an illustration): indeed, for any $k>1$ and any $x\geq 0$
\[\frac{\partial P_k(x)}{\partial k}=4\nu^2dx^2+(4\nu^2(k-1)+4d)x+8(k-1)>0.\] Finally, note that as for $k=2$ $P_2$ only has positive coefficients, $q_2>0$.  We have thus established that $q_1<0$, $q_2>0$, and that $q_k$ is monotone in $k$. As a necessary and sufficient condition for $P_k$ to have two distinct positive real roots is 
\[ P_k(0)>0,\, q_k<0\] the conclusion follows: there exists $k^*\in (1;2)$ such that, for any $k\in (1;k^*)$, $P_k$ has exactly two distinct positive real roots $X_2,\, X_3$.
        \begin{figure}[h!]
            \centering
            \includegraphics[width=0.5\linewidth]{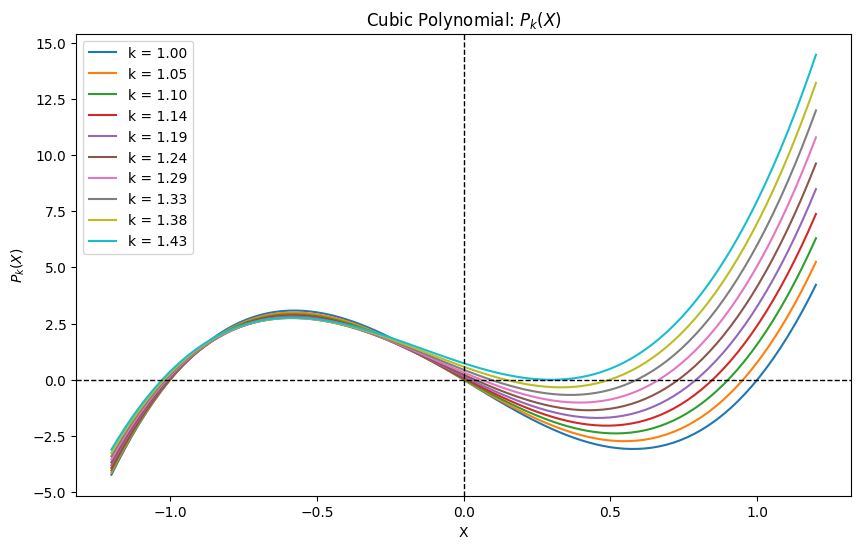}
            \caption{Cubic polynomial $P_{k}(X)$ for different values of $k \in (1, \, k^*)$, $d=2\nu^2>0$.}
            \label{fig:cubic polynomial}
        \end{figure}
    \end{proof}
   We now prove that $d_{\tilde\theta,\tilde\eta}\mathcal F(k,0,0)$ has a one dimensional kernel and Fredholm index 0 provided $\ell$ is chosen correctly.
\begin{proposition}\label{Prop:ker_dimension1}Let $k^*$ be given by Lemma \ref{Lem:properties_of_DeltaP(k)}.
    For any $k\in(1,k^*)$,  let $X_3$ be the largest positive root of the polynomial $P_k$ and define $\ell:=\frac{2\pi}{\sqrt{X_3}}$, $\Omega=(0;\ell)$. There holds:
    \begin{enumerate}
    \item $\ker\left(d_{\tilde\theta,\tilde\eta}\mathcal{F}(k,0,0)\right)$ has dimension $1$.\item $\Im\left(d_{\tilde\theta,\tilde\eta}\mathcal{F}(k,0,0)\right)$ has co-dimension $1$.\end{enumerate}
\end{proposition}

\begin{proof}[Proof of Proposition \ref{Prop:ker_dimension1}]
The fact that $\ker\left(d_{\tilde\theta,\tilde\eta}\mathcal{F}(k,0,0)\right)$ has dimension $1$ when $\ell= \frac{2\pi}{\sqrt{X_3}}$ follows by construction:  first, by construction, $n=1$ is the only solution of $ \det(A_k+n^2\omega^2D_k)=0$. As $X_3$ is a simple root of $P_k$, $\ker(A_k+\omega^2D_k)$ has dimension 1. Fix $(a,b,c)\in \ker(A_k+\omega^2D_k)\setminus\{0\}$ and define $(\xi_0,\zeta_0,\dot\varphi_0):=\left(a,b,c\right)\sin(\omega x)$. It follows that 
\[ \ker\left(d_{\tilde\theta,\tilde\eta} \mathcal{F}(k,0,0)\right)=\R(\xi_0,\zeta_0).\]

To study the co-dimension of the image, note that, in Fourier variables, $d_{\tilde\theta, \tilde\eta}\mathcal{F}(k,0,0)$ decomposes into a block-diagonal family $\left\{ A+\omega^2n^2D\right\}_{n\in\mathbb{N}}$, where each block is a $3\times 3$ matrix. By Lemma \ref{Lem:properties_of_DeltaP(k)} and the fact that, for any $n>1$, $\det(A_k+n^2\omega^2D_k)>0$, the first of those blocks is the only singular one and has a one-dimensional kernel. It follows that its image has co-dimension 1.
\end{proof}

\subsection{Study of the transversality condition}\label{Se:Transversality}
The goal of this section is to establish Point \ref{(iv)Th:localbifurcation} of Theorem \ref{Th:CR}, namely:

    \begin{proposition}\label{Prop:transversalitycondition}
Let $k^*$ be given by Lemma \ref{Lem:properties_of_DeltaP(k)}. Let $k\in (1;k^*)$, and let $X_3$ be the largest positive root of the polynomial $P_k$,  $\ell:=\frac{2\pi}{\sqrt{X_3}}$, $\Omega=(0;\ell)$. Finally, let $(\xi_0,\zeta_0)\in\ker \left(d_{\tilde\theta,\tilde\eta}\mathcal{F}(k,0,0)\right)\setminus\{0\}$. There holds
\[d_{k} d_{\tilde\theta,\tilde\eta} \mathcal{F} (k,0,0) (\xi_0,\zeta_0) \not\in \Im(d_{\tilde\theta,\tilde\eta} \mathcal{F} (k,0,0) ).\]
    \end{proposition}
    
    \begin{proof}[Proof of Proposition \ref{Prop:transversalitycondition}]
    We know that $(\xi_0,\zeta_0,\dot\p_{\xi_0,\zeta_0})=(a,b,c)\sin(\omega \cdot)$, where $(a,b,c)\neq (0,0,0)$ solves
    \[(A_k+\omega^2D_k)\begin{pmatrix}a\\b\\c\end{pmatrix}=\begin{pmatrix}0\\0\\0\end{pmatrix}\] or, equivalently,     \begin{equation}\begin{cases}\label{systemsatisfiedby_abc}
        a(d\omega^2+k-1)+2c(k-1)=0,\\
        b(d\omega^2+k-1)+2\frac{a}{k-1}+2c \frac{2-k}{k-1}=0,\\
        2\nu^2\omega^2c+a\frac{2-k}{k-1}+b (k-1)=0.
    \end{cases}\end{equation} In particular, we deduce that 
    \begin{equation}\label{Eq:LocalInformation}
    \begin{cases}
    ac<0,\, 
          \\  2c= \frac{1}{\nu^2 \omega^2} \left( -\frac{2-k}{k-1}a - (k-1)b\right),
         \\   b = a \frac{\nu^2\omega^2(d\omega^2+k-1)-(2-k)}{(k-1)^2}.
    \end{cases}
    \end{equation}
Now, observe that     
    \[d_{k} d_{\tilde\theta,\tilde\eta} \mathcal{F} (k,0,0) (\xi_0,\zeta_0)= \begin{pmatrix}
        2(k-1) \partial_{k} \dot \varphi +\xi_0 +2\dot\varphi_{\xi_0,\zeta_0}\\
        \zeta_0 -\frac{2}{(k-1)^2}\left(\dot\varphi_{\xi_0,\zeta_0}+\xi_0\right) +2\frac{2-k}{k-1}\partial_{k}\dot\varphi
    \end{pmatrix}\] where
     $\partial_{k} \dot \varphi$ solves \[-2\nu^2 \Delta \partial_{k}\dot\varphi -\frac{\xi_0}{(k-1)^2}+\zeta_0=0,\] so that
     
    \begin{equation*}\label{p_k0dotvarphi_inlocalbifurcationtheorem}
        \partial_{k}\dot\varphi = \frac{1}{2\nu^2\omega^2}\left( \frac{a}{(k-1)^2}-b \right)\sin(\omega\cdot).
    \end{equation*} Consequently,  \begin{align*}
    \partial_{k} d_{\tilde\theta,\tilde\eta} \mathcal{F} (k,0,0) (\xi_0,\zeta_0)
        &= \begin{pmatrix}\frac{k-1}{\nu^2\omega^2}\left( \frac{a}{(k-1)^2}-b\right) +a+2c\\
            b - \frac{2}{(k-1)^2}(c+a) + \frac{2-k}{\nu^2\omega^2(k-1)}\left( \frac{a}{(k-1)^2}-b\right)\end{pmatrix}\sin(\omega\cdot)\\
        &= \begin{pmatrix}
            a\left( 1+\frac{1}{\nu^2\omega^2}\right)-2b \frac{k-1}{\nu^2\omega^2}\\
            -\frac{2a}{(k-1)^3\nu^2\omega^2}\left[ \nu^2\omega^2(k-1)-(2-k)\right] + b \left( 1+\frac{1}{\nu^2\omega^2}\right)\end{pmatrix}\sin(\omega\cdot)
            \\&=M\begin{pmatrix}a\\b\end{pmatrix}\sin(\omega\cdot)
    \end{align*}    with 
    \[ M:= \begin{pmatrix}
        1+\frac{1}{\nu^2\omega^2} & -2\frac{(k-1)}{\nu^2\omega^2}\\
        \frac{-2}{\nu^2\omega^2(k-1)^3} \left[\nu^2\omega^2(k-1)-(2-k)\right] &1+\frac{1}{\nu^2\omega^2}
    \end{pmatrix}.\]
Now, let $(\Phi_0,\Psi_0)\in  \left(\Im \left(d_{\tilde \theta, \tilde \eta} \mathcal{F}(k,0,0)\right)\right)^\perp=\ker(A^*+\omega^2D)$ where $A=\mathrm{diag}(A_k)_{n\in \N\setminus\{0\}}, D=\mathrm{diag}(n^2D_k)_{n\in \N\setminus\{0\}}$. As for any $n>1$ $\det(A_k+n^2\omega D_k)>0$ it follows that $(\Phi_0,\Psi_0,\dot\p_{\Phi_0,\Psi_0})=(\alpha,\beta,\gamma)\sin(\omega\cdot)$, with $(A_k+\omega^2D_k)\begin{pmatrix}\alpha\\\beta\\\gamma \end{pmatrix}=\begin{pmatrix}0\\0\\0\end{pmatrix}$.
In particular, 
 \begin{equation}\label{systemsatisfiedby_alphabetagamma}\begin{cases}
        \alpha \left( d\omega^2+k-1\right) + \frac{2}{k-1}\beta + \frac{2-k}{k-1}\gamma = 0,\\
        \beta \left( d\omega^2+k-1\right) + (k-1)\gamma = 0,\\
        2\nu^2\omega^2\gamma +2(k-1)\alpha + 2 \frac{2-k}{k-1}\beta = 0,
    \end{cases}\end{equation}
which entails
\begin{equation}\label{Eq:PropAlpha}
\begin{cases}
\beta\gamma<0,\, 
\\        \alpha = \beta \frac{\nu^2\omega^2(d\omega^2+k-1)-(2-k)}{(k-1)^2},
\\ \alpha\beta\neq 0,\, 
\\         \gamma= \frac{1}{\nu^2\omega^2}\left( -\alpha(k-1)-\beta\frac{2-k}{k-1}\right).
\end{cases}
\end{equation}
We are now going to show that 
\begin{equation}\label{Eq:Goal}
\left\langle M\begin{pmatrix} a\\ b\end{pmatrix}, \begin{pmatrix}\alpha\\\beta\end{pmatrix}\right\rangle\neq 0,\end{equation} which suffices to conclude. Replacing $\alpha,\, b$ with the values given in \eqref{Eq:LocalInformation}--\eqref{Eq:PropAlpha} we obtain 
\begin{align*}
\left\langle M\begin{pmatrix} a\\ b\end{pmatrix}, \begin{pmatrix}\alpha\\\beta\end{pmatrix}\right\rangle&=2a\beta \frac{\nu^2\omega^2d\omega^2}{\nu^2\omega^2(k-1)^3}\left[-\nu^2\omega^2(d\omega^2+k-1)+3-k\right]
\\&=2a\beta \frac{\nu^2\omega^2d\omega^2}{\nu^2\omega^2(k-1)^3} N(\omega^2),
\end{align*}
where 
\[ N(X)=-\nu^2X(dX+k-1)+3-k.\] As $a,\, \beta\neq 0$ from \eqref{Eq:LocalInformation}--\eqref{Eq:PropAlpha}, it suffices to prove that $N(\omega^2)\neq 0$. However, the euclidean division of $P_k$ by $N$ yields
\[P_k=QN+R \] with 
\[ R(X)=2(k-1)(dX+k+1),\, Q(X)=-2(dX+k-1).\] As $Q<0$ in $[0;+\infty)$, and as $R>0$  in $[0;+\infty)$, the fact that $P_k(\omega^2)=0$ implies $N(\omega^2)\neq 0$, which concludes the proof.\end{proof}
Theorem  \ref{Th:NonUniqueness} then follows by the Crandall-Rabinowitz theorem.

\subsection{Criticality of the bifurcation branch}
Although we cannot fully characterise the bifurcation, we can prove that it is not transcritical (see also Fig.\ref{fig:characterization_bifurcation}).
    \begin{proposition}[Characterization of the bifurcation]\label{Prop:characterizationofbifurcationbranch}
        Let $\epsilon>0$ and consider the following parametrization for the bifurcation branch
        \[\left\{ \sigma\in\R \, : \, \abs{\sigma} < \epsilon \right\} \ni s \longmapsto \left( k(s), \tilde \theta (s), \tilde \eta (s)\right) \in \R \times \left( W^{2,2}\left((0;l)\right)\right)^2,\]
        with $\tilde \theta(s) = \bar \theta(s) -k(s) +1, \, \tilde \eta(s) = \bar \eta(s) - \frac{1}{k(s)-1}$. Then the bifurcation is not transcritical
        \[ k'(0)= 0.\]
    \end{proposition}
    \begin{proof}[Proof of Proposition \ref{Prop:characterizationofbifurcationbranch}]
       Let $k\in (1;k^*)$, and let $X_3$ be the largest positive root of the polynomial $P_k$,  $\ell:=\frac{2\pi}{\sqrt{X_3}}$, $\Omega=(0;\ell)$. Let $\xi_0, \zeta_0 \in \ker \left(d_{\tilde\theta,\tilde\eta}\mathcal{F}(k,0,0)\right)\setminus \left\{0\right\}$ and $\Phi_0, \Psi_0 \in \left( \Im\left(d_{\tilde\theta,\tilde\eta}\mathcal{F}(k,0,0)\right)\right)^\perp$, the following characterizations hold
       \begin{equation}\label{k'(0)}
           k'(0)= -\frac{\left \langle \begin{pmatrix} \Phi_0\\ \Psi_0 \end{pmatrix},\  d^2_{\tilde\theta,\tilde\eta}\mathcal{F}(k,0,0)[\xi_0^2,\zeta_0^2] \right\rangle }{2 \left \langle \begin{pmatrix} \Phi_0\\ \Psi_0 \end{pmatrix},\  \partial_{k}d_{\tilde\theta,\tilde\eta}\mathcal{F}(k,0,0)[\xi_0,\zeta_0] \right\rangle}.
       \end{equation}
       
       Let us observe that the previous quantities are well-defined since, by the transversality condition in Section \ref{Se:Transversality}, we have
       \[ \left \langle \begin{pmatrix} \Phi_0\\ \Psi_0 \end{pmatrix},\ \partial_k d_{\tilde\theta,\tilde\eta}\mathcal{F}(k,0,0)[\xi_0,\zeta_0] \right\rangle \neq 0. \]
       We would like to prove that \[d^2_{\tilde\theta,\tilde\eta}\mathcal{F}(k,0,0)[\xi_0^2,\zeta_0^2] \in \Im\left( d_{\tilde\theta,\tilde\eta}\mathcal{F}(k,0,0)\right).\]
       Differentiating \eqref{dF[xi,zeta]}, we have
       \begin{equation}\label{d2F[xi2,zeta2]}d^2_{\tilde\theta,\tilde\eta}\mathcal{F}(k,0,0)[\xi_0^2,\zeta_0^2] = \begin{pmatrix} 2\xi_0 (2\dot\varphi_{\xi_0,\zeta_0} + \xi_0) + 2(k-1) (\ddot\varphi_{\xi_0,\zeta_0} +\dot\varphi_{\xi_0,\zeta_0}^2)\\ 2\zeta_0 (2\dot\varphi_{\xi_0,\zeta_0}+2\xi_0) + 2\frac{2-k}{k-1}(\ddot\varphi_{\xi_0,\zeta_0} +\dot\varphi_{\xi_0,\zeta_0}^2)\end{pmatrix},
       \end{equation}
       where, by \eqref{Eq:DotPhi}, $\ddot\varphi_{\xi_0,\zeta_0}$ solves
       \begin{equation}\label{equationforddotvarphi0} 
            -2\nu^2\Delta \ddot\varphi_{\xi_0,\zeta_0}+2\xi_0\zeta_0+2\xi_0\dot\varphi_{\xi_0,\zeta_0} \frac{2-k}{k-1}+2(k-1)\zeta_0\dot\varphi_{\xi_0,\zeta_0}=-\ddot{\bar\lambda}_{\xi_0,\zeta_0}, \ \fint_\Omega (\dot\varphi_{\xi_0,\zeta_0}^2+\ddot\varphi_{\xi_0,\zeta_0})=0.
        \end{equation}
        Using the equations satisfied by $\dot\p_{\Phi_0,\Psi_0}$ and $\ddot\varphi_{\xi_0,\zeta_0}$ we obtain
       \begin{small}\begin{align*}
           &\left \langle \begin{pmatrix} \Phi_0\\ \Psi_0 \end{pmatrix},\,  d^2_{\tilde\theta,\tilde\eta}\mathcal{F}(k,0,0)[\xi_0^2,\zeta_0^2] \right\rangle =\\
           &=\fint_0^\ell \left\{\Phi_0 \left[ 2\xi_0 (2\dot\varphi_{\xi_0,\zeta_0} + \xi_0) + 2(k-1) (\ddot\varphi_{\xi_0,\zeta_0} +\dot\varphi_{\xi_0,\zeta_0}^2)\right] + \Psi_0 \left[ 2\zeta_0 (2\dot\varphi_{\xi_0,\zeta_0}+2\xi_0) + 2\frac{2-k}{k-1}(\ddot\varphi_{\xi_0,\zeta_0} +\dot\varphi_{\xi_0,\zeta_0}^2)\right]\right\} \\
           &=\fint_0^\ell \left\{\Phi_0 \left[2\xi_0 (2\dot\varphi_{\xi_0,\zeta_0} + \xi_0) + 2(k-1)\dot\varphi_{\xi_0,\zeta_0}^2\right] + \Psi_0 \left[2\zeta_0 (2\dot\varphi_{\xi_0,\zeta_0}+2\xi_0) + 2\frac{2-k}{k-1} \dot\varphi_{\xi_0,\zeta_0}^2\right]\right\}+\\
           & \qquad +\fint_0^\ell \left[ 2(k-1)\Phi_0\ddot\varphi_{\xi_0,\zeta_0} + 2\frac{2-k}{k-1}\Psi_0 \ddot\varphi_{\xi_0,\zeta_0} \right]\\
           &= \fint_0^\ell \left\{\Phi_0 \left[2\xi_0 (2\dot\varphi_{\xi_0,\zeta_0} + \xi_0) +2 (k-1)\dot\varphi_{\xi_0,\zeta_0}^2\right]+ \Psi_0 \left[2\zeta_0 (2\dot\varphi_{\xi_0,\zeta_0}+2\xi_0) + 2\frac{2-k}{k-1} \dot\varphi_{\xi_0,\zeta_0}^2\right]\right\} +\fint_0^\ell 2\nu^2 \ddot\varphi_{\xi_0,\zeta_0} \Delta \dot\p_{\Phi_0,\Psi_0} \\
           &= \fint_0^\ell \left\{\Phi_0 \left[2\xi_0 (2\dot\varphi_{\xi_0,\zeta_0} + \xi_0) +2 (k-1)\dot\varphi_{\xi_0,\zeta_0}^2\right]+ \Psi_0 \left[2\zeta_0 (2\dot\varphi_{\xi_0,\zeta_0}+2\xi_0) + 2\frac{2-k}{k-1} \dot\varphi_{\xi_0,\zeta_0}^2\right]\right\} + \\
           & \qquad + \fint_0^\ell \left( 2\xi_0\zeta_0 +2\frac{2-k}{k-1} \xi_0\dot\varphi_{\xi_0,\zeta_0}+2(k-1)\zeta_0\dot\varphi_{\xi_0,\zeta_0} +\ddot{\bar\lambda}_{\xi_0,\zeta_0}\right)\dot\p_{\Phi_0,\Psi_0}\\
           &= \fint_0^\ell \left\{\Phi_0 \left[2\xi_0 (2\dot\varphi_{\xi_0,\zeta_0} + \xi_0) +2 (k-1)\dot\varphi_{\xi_0,\zeta_0}^2\right]+ \Psi_0 \left[2\zeta_0 (2\dot\varphi_{\xi_0,\zeta_0}+2\xi_0) + 2\frac{2-k}{k-1} \dot\varphi_{\xi_0,\zeta_0}^2\right] \right\} +\\
           &\qquad + \fint_0^\ell \dot\p_{\Phi_0,\Psi_0} \left[2\xi_0\zeta_0 +2\frac{2-k}{k-1} \dot\p_{\Phi_0,\Psi_0}\dot\varphi_{\xi_0,\zeta_0}+2(k-1)\zeta_0\dot\varphi_{\xi_0,\zeta_0} +\ddot{\bar\lambda}_{\xi_0,\zeta_0}\right].
       \end{align*}\end{small}
       We know that $(\xi_0,\zeta_0,\dot\varphi_{\xi_0,\zeta_0})=(a,b,c)\sin(\omega \cdot)$, where $(a,b,c) \neq (0,0,0)$ solves \eqref{systemsatisfiedby_abc} and $(\Phi_0,\Psi_0,\dot\p_{\Phi_0,\Psi_0})=(\alpha,\beta,\gamma) \sin(\omega \cdot)$ solves \eqref{systemsatisfiedby_alphabetagamma}. Then, we obtain
       \begin{small}\begin{align*}
           &\left \langle \begin{pmatrix} \Phi_0\\ \Psi_0 \end{pmatrix},\,  d^2_{\tilde\theta,\tilde\eta}\mathcal{F}(k,0,0)[\xi_0^2,\zeta_0^2] \right\rangle =\\
           &= \fint_0^\ell \left\{\Phi_0 \left[2\xi_0 (2\dot\varphi_{\xi_0,\zeta_0}+ \xi_0) +2 (k-1)\dot\varphi_{\xi_0,\zeta_0}^2\right]+ \Psi_0 \left[2\zeta_0 (2\dot\varphi_{\xi_0,\zeta_0}+2\xi_0) + 2\frac{2-k}{k-1} \dot\varphi_{\xi_0,\zeta_0}^2\right] \right\}+ \\
           &\qquad + \fint_0^\ell\dot\p_{\Phi_0,\Psi_0} \left[2\xi_0\zeta_0 +2\frac{2-k}{k-1} \xi_0\dot\varphi_{\xi_0,\zeta_0}+2(k-1)\zeta_0\dot\varphi_{\xi_0,\zeta_0}\right]+ \ddot{\bar\lambda}_{\xi_0,\zeta_0}\fint_0^\ell \dot\p_{\Phi_0,\Psi_0}\\
           &= \left\{ \alpha \left[ 2a(2c+a)+2(k-1)c^2\right]+\beta\left[2b(2c+2a)+2\frac{2-k}{k-1}c^2\right]+\gamma \left[ 2ab +2 \frac{2-k}{k-1}ac +2(k-1)bc\right]\right\}\fint_0^\ell \sin^3(\omega x) + \\
           &\qquad + \ddot{\bar\lambda}_{\xi_0,\zeta_0} \gamma \fint_0^\ell\sin(\omega x)\\
           &=0,
       \end{align*}\end{small}
       since $\fint_0^\ell \sin(\omega x)=\fint_0^\ell\sin^3(\omega x)=0$. Hence, by \eqref{k'(0)}, we can conclude that $k'(0)=0$.
   \end{proof}
 
 \section{Uniqueness results: proof of Theorem \ref{Th:Uniqueness}}\label{Sec:ProofTh.Uniqueness}
As we mentioned, it is standard that the Lasry-Lions monotonicity conditions \eqref{static_monotonicity} implies uniqueness of solutions of \eqref{EMFC}, and the proof of Theorem \ref{Th:Uniqueness} relies on the identification of regimes where this condition is satisfied. Keeping in line with \cite{zbMATH08109759}, we establish this monotonicity condition by a fine study of $\bar\theta$ and $\bar \eta$ but, unlike \cite{zbMATH08109759}, in Theorem \ref{Th:Uniqueness}, we provide explicit conditions that $(\ell,k,d,\nu)$ must satisfy to guarantee uniqueness. This section is organised as follows;
\begin{enumerate}
\item In Section \ref{SubSec:generaluniqueness}, we give sufficient analytic conditions on solutions of \eqref{EMFC} to guarantee uniqueness of solutions.
\item In Section \ref{SubSec:PTU}, we use this condition to establish Theorem \ref{Th:Uniqueness}. Let us observe that this condition holds for any $k$ large enough.
\end{enumerate}
 
\subsection{A sufficient condition for uniqueness}\label{SubSec:generaluniqueness}
    The following result  gives a general, although at first glance hard to check, sufficient condition for uniqueness:  \begin{theorem}\label{Th:generaluniqueness} Let $\mathscr S$ denote the set of solutions of \eqref{EMFC}. If, for any $(\bar\theta,\bar\eta,\bar\lambda,\bar\p)\in \mathscr S$ there holds        \begin{equation}\label{keyconditionasinBHL}
  \forall x \in \O,\        d\frac{\abs{\nabla \left( \frac{\bar \theta}{1-\bar \eta}\right)}^2}{\left( \frac{\bar \theta}{1-\bar \eta}\right)^3} (x)  \le 4        \end{equation}
        then the Lasry-Lions monotonicity condition
    \begin{equation}\label{static_monotonicity}\tag{$LL_S$}
        \int_\Omega \left[\bar \theta_1(1-\bar \eta_1)-\bar \theta_2 (1-\bar\eta_2)\right] (\bar\p_1^2 -\bar\p_2^2) \le 0,
    \end{equation} where $(\bar\theta_i,\bar\eta_i,\bar\lambda_i,\bar\p_i)\in \mathscr S,\, i=1,2$
    is satisfied, and so \eqref{EMFC} has a unique solution.
    \end{theorem}

    \begin{proof}[Proof of Theorem \ref{Th:generaluniqueness}]
        Let $(\bar\theta_i,\bar\eta_i,\bar\lambda_i,\bar\p_i)\in \mathscr S,\, i=1,2$  and define, for any $\tau \in [0;1]$,    \[ \bar\p^2_\tau = \tau \bar\p^2_1 +(1-\tau)\bar\p^2_2.\]
    Consider the map
    \[ A:[0;1]\ni\tau\mapsto \int_\Omega \left[ \bar \theta_{\bar\p^2_\tau}(1-\bar \eta_{\bar\p^2_\tau}) - \bar \theta_2 (1-\bar \eta_2)\right] (\bar\p^2_1 - \bar\p^2_2)\] where, for any $\tau \in [0;1]$ and any probability measure $\mu\in L^2(\O)$, $\bar\theta_{\mu}$ denotes the unique solution of
    \[\begin{cases}-d\Delta\bar\theta_\mu-\bar\theta_\mu(k-\mu-\bar\theta_\mu)=0&\text{ in }\O,\, 
    \\ \partial_\nu \bar\theta_\mu=0&\text{ on }\partial \O,\, 
    \\ \bar\theta_\mu>0&\text{ in }\O.\end{cases}\]
Observe that $A(0)=0$ and that \eqref{static_monotonicity} is equivalent to $A(1)\leq 0$, so that our strategy amounts to showing that \eqref{keyconditionasinBHL} implies
\begin{equation}\label{Eq:Adecreasing}
A'\leq 0\text{ in }[0;1].
\end{equation} Let, for any probability measure $\mu\in L^2(\O)$ and any $h\in L^2(\O)$, $(\dot{\bar\theta}_\mu[h],\dot{\bar\eta}_\mu[h])$ denote the Gateaux derivative of $\mu\mapsto(\bar\theta_\mu,\bar\eta_\mu)$ at $\mu$ in the direction $h$ or, in other words, the unique solution of 
\begin{equation}\label{Eq:DotTheta}
\begin{cases}
-d\Delta \dot{\bar\theta}_\mu[h]-  \dot{\bar\theta}_\mu[h]\left(k-\mu-2 \bar\theta_{\mu}\right)=-h\bar\theta_\mu&\text{ in }\O,\, 
\\ -d\Delta  \dot{\bar\eta}_\mu[h]- \dot{\bar\eta}_\mu[h]\left(k-\mu-2 \bar\theta_\mu\right)=h(1-\bar\eta_\mu)-2\dot{\bar\theta}_\mu\bar\eta_\mu&\text{ in }\O,\, 
\\ \partial_\nu\dot{\bar\theta}_\mu[h]= \partial_\nu\dot{\bar\eta}_\mu[h]=0&\text{ on }\partial \O.
\end{cases}\end{equation}
We deduce that, for any $\tau\in [0;1]$, 
\[ A'(\tau)=\int_\O \left(\dot{\bar\theta}_{\bar\p_\tau^2}[\bar\p_1^2-\bar\p_2^2](1-\bar\eta_{\bar\p_\tau^2})-\dot{\bar\eta}_{\bar\p_\tau^2}[\bar\p_1^2-\bar\p_2^2]\bar\theta_{\bar\p_\tau^2}\right) \left(\bar\p_1^2 -\bar\p_2^2\right)
\] or, introducing $h:=\bar\p_1^2-\bar\p_2^2$ and setting for notational convenience $\dot{\bar\theta}:=\dot{\bar\theta}_{\bar\p_\tau^2}[\bar\p_1^2-\bar\p_2^2]$, $\dot{\bar\eta}:=\dot{\bar\eta}_{\bar\p_\tau^2}[\bar\p_1^2-\bar\p_2^2]$,
\begin{equation}\label{Eq:Aprime}
A'(\tau)=\int_\O  \left( \dot{\bar\theta} (1-\bar\eta) -\bar\theta \dot{\bar\eta}\right) h.
\end{equation}
Now, observe that from \eqref{Eq:DotTheta} there holds
\[ -\int_\O h\bar\theta_{\bar\p_\tau^2}\dot{ \bar\eta}=\int_\O h\dot{\bar\theta}(1-\bar\eta_{\bar\p_\tau^2})-2\int_\O \dot{\bar\theta}^2\bar\eta_{\bar\p_\tau^2}\leq \int_\O h\dot{\bar\theta}(1-\bar\eta_{\bar\p_\tau^2})
\]
whence 
\begin{equation}\label{Eq:DotA}
\frac12A'(\tau)= \int_\O h\dot{\bar\theta}(1-\bar\eta_{\bar\p_\tau^2})-\int_\O \dot{\bar\theta}^2\bar\eta_{\bar\p_\tau^2}.
\end{equation} Set \[ p:=\frac{\bar\theta_{\bar\p_\tau^2}}{1-\bar\eta_{\bar\p_\tau^2}}\] so that
\[ -\int_\O h\dot{\bar\theta}(1-\bar\eta_{\bar\p_\tau^2})=-\int_\O (h\bar\theta_{\bar\p_\tau^2})\frac{\dot{\bar\theta}}{p}\]
Now, write 
\[ -h\bar\theta_{\bar\p_\tau^2}={-d\Delta\dot{\bar\theta}-\dot{\bar\theta}(k-\bar\p_\tau^2-2\bar\theta_{\bar\p_\tau^2})}\] so that, setting 
\[ W:=k-\bar\p_\tau^2-2\bar\theta_{\bar\p_\tau^2},\] we derive 
\begin{align*}
-\int_\O h\dot{\bar\theta}(1-\bar\eta_{\bar\p_\tau^2})&=-d\int_\O \frac{\dot{\bar\theta}}{p}\Delta\dot{\bar\theta}-\int_\O \frac{\dot{\bar\theta}^2}{p}W
\\&=-d\int_\O \frac{\dot{\bar\theta}}{\sqrt{p}}\cdot\frac{\Delta\dot{\bar\theta}}{\sqrt{p}}-\int_\O\left( \frac{\dot{\bar\theta}}{\sqrt{p}}\right)^2W.
\end{align*}
As 
\[ \Delta\left(\frac{\dot{\bar\theta}}{\sqrt{p}}\right)=\frac{\Delta\dot{\bar\theta}}{\sqrt{p}}+\Delta\left(\frac1{\sqrt{p}}\right)\dot{\bar\theta}+2\left\langle \n\dot{\bar\theta},\n\left(\frac1{\sqrt{p}}\right)\right\rangle\] we deduce 
\begin{align*}-d\int_\O \frac{\dot{\bar\theta}}{\sqrt{p}}\cdot\frac{\Delta\dot{\bar\theta}}{\sqrt{p}}&=d\int_\O\left|\n\left(\frac{\dot{\bar\theta}}{\sqrt{p}}\right)\right|^2
\\&+d\int_\O\left\{ \frac{\dot{\bar\theta}^2}{\sqrt{p}}\Delta\left(\frac1{\sqrt{p}}\right)+2\frac{\dot{\bar\theta}}{\sqrt{p}}\left\langle \n\dot{\bar\theta},\n\left(\frac1{\sqrt{p}}\right)\right\rangle\right\}
\\&=d\int_\O\left\{\left|\n\left(\frac{\dot{\bar\theta}}{\sqrt{p}}\right)\right|^2+d\n\cdot\left(\frac{\dot{\bar\theta}^2}{\sqrt{p}}\n\left(\frac1{\sqrt{p}}\right)\right)-d\dot{\bar\theta}^2\left|\n\left(\frac1{\sqrt{p}}\right)\right|^2\right\}
\\&=d\int_\O \left|\n\left(\frac{\dot{\bar\theta}}{\sqrt{p}}\right)\right|^2-d\int_\O \dot{\bar\theta}^2\left|\n\left(\frac1{\sqrt{p}}\right)\right|^2.
\end{align*}
Consequently, 
\begin{align*}
-\int_\O h\dot{\bar\theta}(1-\bar\eta_{\bar\p_\tau^2})+\int_\O \dot{\bar\theta}^2\bar\eta_{\bar\p_\tau^2}&=-d\int_\O \dot{\bar\theta}^2\left|\n\left(\frac1{\sqrt{p}}\right)\right|^2
\\&+d\int_\O \left|\n\left(\frac{\dot{\bar\theta}}{\sqrt{p}}\right)\right|^2-\int_\O W\left( \frac{\dot{\bar\theta}}{\sqrt{p}}\right)^2+\int_\O \dot{\bar\theta}^2\bar\eta_{\bar\p_\tau^2}
\\&=-d\int_\O \dot{\bar\theta}^2\left|\n\left(\frac1{\sqrt{p}}\right)\right|^2+\int_\O \bar\theta_{\bar\p_\tau^2}\left( \frac{\dot{\bar\theta}}{\sqrt{p}}\right)^2+\int_\O \dot{\bar\theta}^2\bar\eta_{\bar\p_\tau^2}
\\&+d\int_\O \left|\n\left(\frac{\dot{\bar\theta}}{\sqrt{p}}\right)\right|^2-\int_\O (k-\bar\p_\tau^2-\bar\theta_{\bar\p_\tau^2})\left( \frac{\dot{\bar\theta}}{\sqrt{p}}\right)^2
\\&\geq -d\int_\O \dot{\bar\theta}^2\left|\n\left(\frac1{\sqrt{p}}\right)\right|^2+\int_\O \bar\theta_{\bar\p_\tau^2}\left( \frac{\dot{\bar\theta}}{\sqrt{p}}\right)^2+\int_\O \dot{\bar\theta}^2\bar\eta_{\bar\p_\tau^2}
\\&= \int_\O \dot{\bar\theta}^2\left(1-d\left|\n\left(\frac1{\sqrt{p}}\right)\right|^2\right)
\\&=\int_\O \dot{\bar\theta}^2\left(1-\frac{d}4\cdot\frac{|\n p|^2}{p^3} \right),
\end{align*}
and so 
\[ \frac12A'(\tau)\leq -\int_\O \dot{\bar\theta}^2\left(1-\frac{d}4\cdot\frac{|\n p|^2}{p^3} \right).\] The conclusion follows.
     \end{proof}

\subsection{Proof of Theorem \ref{Th:Uniqueness}}\label{SubSec:PTU}
Theorem \ref{Th:Uniqueness} follows from the following proposition:
\begin{proposition}\label{Prop:Intermediate}
    For any $C>3$ assume $k\ge \max\left\{k_2,k_3\right\}$, with $k_2$ defined as in \eqref{condition_C_K} and 
    {\begin{equation}\label{assumptionfor_theta/1-eta_XiBaiLi}
    k_3:=C \sqrt{A(C)} \left(\frac{3\ell}{2d} \frac{3C-2}{(C-2)^2(\sqrt{6}-1)^2} + \frac{C}{C-2}\sqrt{A(C)}\right),  
\end{equation}
where $A(C):=1+\frac{\ell^2(3C-4)}{\nu^2(C-2)} +\frac{\ell}{\nu^2}\sqrt{\frac{\ell^2(3C-4)^2}{\nu^2(C-2)^2}+8\frac{C-1}{C-2}}$.} Then \begin{equation*}
    \begin{cases}
        1-\bar\eta>0&\text{ in } (0;\ell),\, 
 \\             d\frac{\left|{\left( \frac{\bar \theta}{1-\bar \eta}\right)'}\right|^2}{\left( \frac{\bar \theta}{1-\bar \eta}\right)^3}<4 &\text{ in }(0;\ell) 
    \end{cases}
\end{equation*}
holds, so that \eqref{EMFC} has a unique solution.\\
\end{proposition}
{Note that $A'(C)<0$ and $ \frac{C(3C-2)}{(C-2)^2}<0$ for any $C>3$, while $\left(\frac{C^2}{C-2}\right)'=\frac{C(C-4)}{(C-2)^2}$ reaches its minimum at $C=4$. Then for $\frac{\ell}{d}, \frac{\ell}{\nu} \to 0^+$, $C_\text{opt}=4$, otherwise $C_\text{opt}=C^*_{\ell,d,\nu}$.}\\

The proof of Proposition   \ref{Prop:Intermediate} is lengthy and relies on fine estimates for $\bar\theta,\bar\eta$, similar to those derived by Bai, He \& Li \cite{BHL2015}. Throughout, $(\bar\theta,\bar\eta,\bar\lambda,\bar\p)$ denotes a solution of \eqref{EMFC}.

\begin{lemma}\label{Prop:lambda}
    Assume $k\ge \frac{3+\sqrt{5}}{2}$, then
    \begin{equation}\label{Eq:Ide4}
        -\bar\lambda \le M_{\nu,d}k, \quad {where} \quad M_{\nu,d}:= 2 \left[\frac{5\nu^2}{d}(d+1) +1\right].
    \end{equation}
\end{lemma}
\begin{proof}[Proof of Lemma \ref{Prop:lambda}]
First of all, from 
\[\begin{cases}
    -d\theta'' -\theta (k-\bar\p^2-2\theta) = \theta^2,\\
    -d\eta'' -\eta (k-\bar\p^2-2\theta)=\p^2
\end{cases}\]
we deduce
\begin{equation}\label{Eq:Ide1}
 \int_0^\ell \theta \bar\p^2=\int_0^\ell \eta \theta^2.
 \end{equation}
 Second, observe that 
 \[ -d\theta''+\theta(\bar\p^2+\theta)=k\theta.\] Consequently, 
 \[ k=\min_{u\, ; \int_0^\ell u^2=1}d\int_0^\ell (u')^2+\int_0^\ell u^2(\bar\p^2+\theta),\] which implies 
 \[ k\leq \fint_0^\ell (\bar\p^2+\theta), \]
 whence 
 \[ \fint_0^\ell \theta\geq k-1.\] By the Cauchy-Schwarz inequality and since $k\ge \frac{3+\sqrt{5}}{2}$, we deduce  
 \begin{equation}\label{Eq:Ide2}
 \fint_0^\ell \theta^2\ge \left(\fint_0^\ell \theta\right)^2\ge (k-1)^2 \ge k.\end{equation}
 Third, note that 
 \[ d\int_0^\ell (\theta')^2=\int_0^\ell k\theta^2-\int_0^\ell(\bar\p^2+\theta)\theta^2\] whence 
 \begin{equation}\label{Eq:Ide3}
 \int_0^\ell (\theta')^2\le \frac{k}{d}\int_0^\ell \theta^2.
 \end{equation}
 Finally, let us show \eqref{Eq:Ide4}. To this end, observe that for any $u\in W^{1,2}((0;\ell))\setminus\{0\}$ there holds 
 \begin{equation}\label{Eq:Rayleigh} -\lambda\leq \frac{2\nu^2\int_0^\ell (u')^2-\int_0^\ell u^2\theta(1-\eta)}{\int_0^\ell u^2}.\end{equation}
 Let $u:=\frac{\theta}{\sqrt{1+\eta}}$. 
Note that 
\[\int_0^\ell u^2=\int_0^\ell \frac{\theta^2}{1+\eta}.\]  Now, note that 
\begin{align*}
 \int_0^\ell \theta^2
&\le \left(\int_0^\ell \theta^2(1+\eta)\right)^{\frac12}\left(\int_0^\ell \frac{\theta^2}{1+\eta}\right)^{\frac12}
\\&\le \left(\int_0^\ell \theta^2+\int_0^\ell \theta^2\eta\right)^{\frac12}\left(\int_0^\ell \frac{\theta^2}{1+\eta}\right)^{\frac12}.
\end{align*}
From \eqref{Eq:Ide1}
\[ \int_0^\ell \theta^2\eta=\int_0^\ell \bar\p^2\theta {\leq}\ell k\le  \int_0^\ell \theta^2.\]
Thus
\begin{align*}
\int_0^\ell \theta^2&\le \sqrt{2}\left(\int_0^\ell \theta^2\right)^{\frac12}\left(\int_0^\ell \frac{\theta^2}{1+\eta}\right)^{\frac12}
\\&=\sqrt{2}\Vert \theta\Vert_{L^2} \Vert u\Vert_{L^2}
\end{align*} and thus
\begin{equation}\label{Eq:uL2}
\int_0^\ell u^2\ge \frac{1}{2}\int_0^\ell\theta^2.
\end{equation}
Furthermore, 
\begin{align*}
\int_0^\ell u^2\theta(1-\eta)&=\int_0^\ell \frac{\theta^3 (1-\eta)}{1+\eta}
\\&\geq -\int_0^\ell \frac{\theta^3\eta}{1+\eta}
\\&\geq -\int_0^\ell \theta^3
\\&\geq -k\int_0^\ell \theta^2
\end{align*}
so that 
\begin{equation}\label{Eq:uPotential}
-\int_0^\ell u^2\theta(1-\eta)\le k\int_0^\ell \theta^2.
\end{equation}
Finally, we need to estimate $\int_0^\ell (u')^2$. To this end, by the Young inequality, we obtain 
\begin{align*}
\int_0^\ell (u')^2&=\int_0^\ell \left(\frac{\theta'}{\sqrt{1+\eta}}-\frac12\frac{\theta\eta'}{(1+\eta)^{\frac32}}\right)^2
\\&\le \frac{5}{4}\left(\int_0^\ell \frac{(\theta')^2}{(1+\eta)}+\int_0^\ell \frac{\theta^2(\eta')^2}{(1+\eta)^3}\right)
\\&\le \frac{5}{4}\left(\int_0^\ell (\theta')^2+ \int_0^\ell \frac{\theta^2(\eta')^2}{(1+\eta)^3}\right)
\\&\le \frac{5}{4}\left( \frac{k}{d}\int_0^\ell \theta^2+\int_0^\ell \frac{\theta^2(\eta')^2}{(1+\eta)^3}\right).
\end{align*}
Using $v:=\frac{\theta^2}{(1+\eta)^2}$ as a test function in the $\eta$ equation, we obtain 
\begin{align*}
-2\int_0^\ell \frac{(\eta')^2\theta^2}{(1+\eta)^3}+2\int_0^\ell \frac{\eta'\theta'\theta}{(1+\eta)^2}&=\int_0^\ell \bar\p^2\frac{\theta^2}{(1+\eta)^2}+\int_0^\ell \frac{\theta^2\eta}{(1+\eta)^2} \cdot (k-\bar\p^2-2\theta)
\end{align*}
so that, from the Young inequality
\begin{align*}
2\int_0^\ell \frac{(\eta')^2\theta^2}{(1+\eta)^3}&=2\int_0^\ell \frac{\eta'\theta'\theta}{(1+\eta)^2}-\int_0^\ell \bar\p^2\frac{\theta^2}{(1+\eta)^2}-\int_0^\ell \frac{\theta^2\eta}{(1+\eta)^2} \cdot (k-\bar\p^2-2\theta)
\\&\leq  \int_0^\ell \frac{(\eta')^2\theta^2}{(1+\eta)^3}+ \int_0^\ell \frac{(\theta')^2}{(1+\eta)}+\int_0^\ell \theta^2 \frac{\eta}{(1+\eta)^2}(\bar\p^2+\theta)
\end{align*}
whence, using
\[ \frac{\eta}{(1+\eta)^2}\le \frac{1+\eta}{(1+\eta)^2}\le 1, \]
we obtain
\begin{align*}
\int_0^\ell \frac{(\eta')^2\theta^2}{(1+\eta)^3}&\le \int_0^\ell (\theta')^2+\int_0^\ell \theta^2\bar\p^2+\int_0^\ell\theta^3
\\&\le \frac{k}{d}\int_0^\ell \theta^2+k^2 \int_0^\ell \bar\p^2 + k \int_0^\ell \theta^2
\\&\le k\left( \frac{1}{d}+2\right)\int_0^\ell \theta^2\text{ from \eqref{Eq:Ide2}}.
\end{align*}
Overall, we have thus derived 
\begin{equation}\label{Eq:uGradient} \int_0^\ell (u')^2\le \frac{5}{2} \left(\frac{1}{d}+1\right) k\int_0^\ell \theta^2.\end{equation}
Combining \eqref{Eq:uL2}--\eqref{Eq:uPotential}--\eqref{Eq:uGradient} with \eqref{Eq:Rayleigh} yields 
\begin{align*}
-\lambda&\leq \frac{ 2k\left(5\nu^2 \left(\frac{1}{d}+1\right)+1\right)\int_0^\ell \theta^2}{\int_0^\ell \theta^2}= M_{\nu,d} k.
\end{align*}
This concludes the proof of \eqref{Eq:Ide4}.
\end{proof}

We now derive Bai, He \& Li type estimates \cite{BHL2015} for $\frac{\abs{\nabla\bar\theta}^2}{\bar\theta^3}$.
    \begin{lemma}\label{Prop:BHL_bartheta}
Assume $k\ge \frac{3+\sqrt{5}}{2}$. Then         \begin{equation}\label{upperboundforvarphi2}
             \|\bar\p^2\|_{L^\infty} \le 1+ \frac{\sqrt{2}\ell}{\nu}\sqrt{\left(M_{\nu,d} + 1\right) k}.
        \end{equation}
Furthermore, if $k \ge 1+ \frac{\sqrt{2}\ell}{\nu}\sqrt{\left(M_{\nu,d} + 1\right)k}$, namely
        \begin{equation}\label{k-varphi^2>0}
           k\ge\left(\frac{\ell}{\sqrt{2}\nu}\sqrt{M_{\nu,d} + 1} + \sqrt{\frac{\ell^2}{2\nu^2}\left(M_{\nu,d} + 1\right)+1}\right)^2,
        \end{equation}
        then
        \begin{equation}\label{BHL_bartheta}
            d\frac{\abs{\bar\theta'}^2}{\bar\theta^3} < \frac{2}{3}.
        \end{equation}
    \end{lemma}
    \begin{proof}[Proof of Lemma \ref{Prop:BHL_bartheta}] 
    By Lemma \ref{Prop:lambda}, we have $-\bar\lambda \le M_{\nu,d} k$, then
    \[ 2\nu^2\fint_0^\ell \abs{\bar\p'}^2= \fint_0^\ell \left( \bar\theta(1-\bar\eta)-\bar\lambda\right) \bar\p^2 \le \max \left( \bar\theta -\bar\lambda\right) \le \left( M_{\nu,d} +1\right) k,\]
    and, together with the Poincar\'e-Wirtinger inequality, we obtain
    \begin{align*}
        \abs{\bar\p^2 -\fint_0^\ell \bar\p^2}&\le \int_0^\ell \abs{(\bar\p^2)'}\\
        &= 2\ell \fint_0^\ell \abs{\bar\p\bar\p'}\\
        &\le 2\ell \sqrt{\fint_0^\ell \bar\p'^2}\\
        &\le \frac{\sqrt{2}\ell}{\nu} \sqrt{\left( M_{\nu,d} +1\right)k},
    \end{align*}
    which implies
    \begin{equation}\label{Linfty_bound_barphi}
        \bar \p^2 \le 1+\frac{\sqrt{2}\ell}{\nu}\sqrt{\left(M_{\nu,d}+1\right)k}.
    \end{equation}
    Now assume that \eqref{k-varphi^2>0} holds, so that the previous estimate entails
        \[ \bar\p^2 \le k\]
    Let us define $v:=d(\bar\theta')^2 -\alpha \bar\theta^3$, with $\alpha>\frac23$, and let us prove that $v\leq 0$. Arguing by contradiction, let $\bar x$ be a maximiser of $v$, with $v(\bar x)>0$. As $v<0$ at $x=0,\, \ell$, $\bar x$ is an interior point, and $\bar\theta'(\bar x)>0$. We thus derive
    \[ 0=v'(\bar x) =\bar\theta'(\bar x) \left( 2d\bar\theta''(\bar x) -3 \alpha \bar\theta^2(\bar x)\right),\]
    which, as $\bar \theta'(\bar x) \neq 0$ and $\bar\p^2(\bar x) \le  k$, gives
    \begin{align*} 
        \frac{3}{2}\alpha \bar \theta(\bar x)=d \frac{\bar\theta''(\bar x)}{\bar\theta(\bar x)}
        = \bar \theta(\bar x) -k + \bar\p^2(\bar x) 
        \le \bar \theta(\bar x),
        \end{align*}
    a contradiction as $\frac32\alpha>1$.
\end{proof}
         
The following Lemma provides uniform estimates in the parameter $k$.        
\begin{lemma}\label{Prop:1-bareta}
    Assume $k\ge \frac{3+\sqrt{5}}{2}$ and $1+\frac{\sqrt{2}\ell}{\nu}\sqrt{\left( M_{\nu,d}+1\right)k}\le \frac{k}{C}$ for any $C>2$, namely 
    \begin{equation}\label{condition_C_K}
        k\ge k_2 ,\quad k_2:=\frac{C^2}{4} \left[ \frac{\sqrt{2}\ell}{\nu}\sqrt{M_{\nu,d}+1} + \sqrt{\frac{2\ell^2}{\nu^2}\left( M_{\nu,d}+1\right) + \frac{4}{C}}\right]^2.
    \end{equation}
    Then $\bar\theta >\frac{k}{2}$ and $1-\bar\eta>0$. Moreover, we can derive the following bounds
    \begin{equation}\label{Eq:GdBound}
        \Vert \bar\p'\Vert_{L^2} \le {\frac{\ell^\frac32(3C-4)}{2\nu^2(C-2)}+\frac{\sqrt{\ell}}{2\nu^2}\sqrt{\frac{\ell^2(3C-4)^2}{\nu^2(C-2)^2}+ \frac{8(C-1)}{C-2}}},
    \end{equation}
    \begin{equation}\label{Eq:GdBound'}
        \Vert \bar\p'\Vert_{L^\infty}\le{\frac{\ell^2(3C-4)}{2\nu^2(C-2)}+\frac{\ell}{2\nu}\sqrt{\frac{\ell^2(3C-4)^2}{\nu^2(C-2)^2}+ \frac{8(C-1)}{C-2}}},
    \end{equation}
    \begin{equation}\label{Eq:Dubeau3} \Vert \bar\theta'\Vert_{L^\infty(0;\ell)} \le \frac{2C}{C-2}\Vert \bar\p\Vert_{L^\infty} \Vert \bar \p'\Vert_{L^\infty}\end{equation}
    and
    \begin{equation}\label{Eq:eta'_L^infty}
        \Vert \bar\eta' \Vert_{L^\infty} \le \frac{\ell}{d} \Vert \bar\p \Vert_{L^\infty} \frac{3C-2}{C-2}.        
    \end{equation}
\end{lemma}
\begin{proof}
    Recall that, as we are working in one dimension, there holds
    \begin{equation}\label{Eq:bound_phi_L^infty}
        \Vert \bar\p\Vert_{L^\infty}^2 \le 1+ 2\sqrt{\ell}\Vert \bar\p'\Vert_{L^2}.
    \end{equation}
    Integrate the equation on $\bar\eta$, which gives
    \[1= \fint_0^\ell \left[2(\bar\theta -k)+\bar\p^2\right]\bar\eta +k \fint_0^\ell \bar\eta.\]
    Observe that, if \eqref{condition_C_K} holds, then \eqref{Linfty_bound_barphi}--\eqref{condition_C_K} imply $k-2\Vert \bar\p\Vert^2_{L^\infty} \ge k \frac{C-2}{C}>0$. {On the other hand, the maximum principle yields
    \begin{equation}\label{Eq:Ide0}
        0 \le k -\bar\theta \le \Vert \bar\p\Vert^2_{L^\infty},
    \end{equation}}
    and so
    \begin{equation}\label{Eq:fint_eta}
        \fint_0^\ell \bar\eta \le \frac{C}{k(C-2)}.
    \end{equation}
    We can now conclude our bootstrap approach; once more by the Rayleigh quotient formulation of $\bar\lambda$, we have
    \[ -\bar\lambda \le -\fint_0^\ell \bar\theta(1-\bar\eta),\]
    whence
    \begin{align*}
        2\nu^2\fint_0^\ell (\bar\p')^2 &= -\bar\lambda+\fint_0^\ell \bar\p^2 \bar\theta(1-\bar\eta)\\
        &\le \fint_0^\ell (\bar\p^2-1) \bar\theta(1-\bar\eta)\\
        &= \fint_0^\ell \bar\theta \left( \bar\p^2 -1\right) -\fint_0^\ell \bar\theta\bar\eta \left( \bar\p^2-1\right)\\
        &\le {\fint_0^\ell (\bar\theta -k) (\bar\p^2-1)} + \left( 1+ \Vert \bar\p\Vert^2_{L^\infty}\right) \fint_0^\ell \bar\theta\bar\eta \\
        &\le {2\Vert \bar\p\Vert^2_{L^\infty}} + \left( 1+ \Vert \bar\p\Vert^2_{L^\infty}\right)  \fint_0^\ell \bar\theta\bar\eta \text{ by \eqref{Eq:Ide0}}.       
    \end{align*}
    Observe that
    \[ \fint_0^\ell \bar\theta \bar\eta \le k \fint_0^\ell \bar\eta \le\frac{C}{C-2}\]
    by \eqref{Eq:fint_eta}. Overall, we thus deduce 
    \[ \Vert\bar\p'\Vert^2_{L^2}\le {\frac{\ell}{\nu^2}\left( 2\frac{C-1}{C-2}+ \frac{3C-4}{C-2}\sqrt{\ell}\Vert\bar\p'\Vert_{L^2}\right)}, \]
    thereby concluding the proof of \eqref{Eq:GdBound}.\\
    \\
    
    Now, observe that \eqref{condition_C_K} also entails
    \begin{equation}\label{Eq:Etak2} 
        0 < \bar\eta <1.
    \end{equation}
{Indeed, we deduce from \eqref{Eq:Ide0}}:
    \begin{equation}\label{2theta+phi^2-k_bounds}
        0< k \frac{C-2}{C} \le 2\bar\theta +\bar\p^2 -k \le 2k,
    \end{equation}
    whence $\bar\theta >\frac{k}{2}$. Direct computations show that $1-\bar\eta$ solves
    \[ -d(1-\bar\eta)'' -(1-\bar\eta)(k-\bar\p^2-2\bar\theta)=2\bar\theta -k >0.\]
    It follows from the maximum principle that $1-\bar\eta>0,$ thereby concluding the proof.\\
    But this, in turn, implies \eqref{Eq:GdBound'}. Indeed, from the Rayleigh quotient, we obtain
    \begin{align*}
        0&\le -\bar\lambda +k\\
        &\le -\fint_0^\ell \bar\theta(1-\bar\eta) +k\\
        &= \fint_0^\ell (k-\bar\theta) +\fint_0^\ell \bar\theta\bar\eta\\
        &\le 1+2\sqrt{\ell} \Vert \bar\p'\Vert_{L^2} +\fint_0^\ell \bar\theta\bar\eta \text{ by \eqref{Eq:Ide0}--\eqref{Eq:bound_phi_L^infty}}\\
        & \le {1+ \frac{\ell^2(3C-4)}{\nu^2(C-2)}+\frac{\ell}{\nu}\sqrt{\frac{\ell^2(3C-4)^2}{\nu^2(C-2)^2}+ \frac{8(C-1)}{C-2}}} + \frac{C}{C-2} \text{ by \eqref{Eq:GdBound}--\eqref{Eq:fint_eta}.}
    \end{align*}
    Now recall that
    \[ -2\nu^2\bar\p'' = \left(-\bar\lambda+k\right) \bar\p - \left( k-\bar\theta(1-\bar\eta)\right) \bar\p.\]
    As both $-\bar\lambda+k$ and $k-\bar\theta(1-\bar\eta)$ are (pointwise) uniformly bounded in $L^2$ and since $\Vert \bar\p\Vert_{L^2}=\sqrt{\ell}$, we deduce 
    \[ \Vert \p''\Vert_{L^2(0;\ell)} \le \frac{\sqrt{\ell}}{2\nu^2}\left( -\bar\lambda +k\right),\]
    and, from Sobolev embedding we can conclude that
    \begin{align*}
        \Vert \bar\p'\Vert_{L^\infty}&\le \sqrt{\ell}\|\bar\p''\|_{L^2}\\
        &\le  {\frac{\ell^2(3C-4)}{2\nu^2(C-2)}+\frac{\ell}{2\nu}\sqrt{\frac{\ell^2(3C-4)^2}{\nu^2(C-2)^2}+ \frac{8(C-1)}{C-2}}}.
    \end{align*} 
    
    Let us show \eqref{Eq:Dubeau3}. To this end, let $q:=\bar\theta'$, and observe that $q$ solves
    \[-dq''+q(2\bar\theta+\bar\p^2-k)=-2\bar\theta\bar\p'\bar\p.\] Let $x_0$ be a point of non-positive minimum (since $ q'(0)=q'(\ell)=0$) of $q$. We deduce that 
    \[ - q(x_0) k \frac{C-2}{C} \le -q(x_0)(2\bar\theta+\bar\p^2-k)\le 2k \Vert \bar\p\Vert_{L^\infty} \Vert \bar \p'\Vert_{L^\infty}\]  
    by \eqref{condition_C_K}, which implies
    \[ q(x) \ge q(x_0) \ge - \frac{2C}{C-2}\Vert \bar\p\Vert_{L^\infty} \Vert \bar \p'\Vert_{L^\infty}\quad \forall x \in (0,\ell).\]
    Similarly, let $x_1$ be a point of positive maximum of $q$; we deduce that 
    \[ q(x) \le q(x_1)\le \frac{2C}{C-2} \Vert \bar\p\Vert_{L^\infty} \Vert \bar \p'\Vert_{L^\infty} \] which yields the same conclusion, whence 
    \eqref{Eq:Dubeau3} follows.\\
    \\

 Let us now prove \eqref{Eq:eta'_L^infty}. Recall that for $k\ge k_2$ we have \eqref{2theta+phi^2-k_bounds}, namely
 \[k-\bar\p^2-2\bar\theta\le - k \frac{C-2}{C}\] whence 
 \[-d\bar\eta''+\frac{C-2}{C} k\bar\eta\leq \bar\p^2\le \|\bar\p\|^2_{L^\infty}. \] In particular, by \eqref{Eq:fint_eta}
 \begin{equation}\label{Eq:Dubeau2}
    \fint_0^\ell \bar\eta^2\le \frac{\Vert \bar \p\Vert^2_{L^\infty}}{k^2} \left( \frac{C}{C-2}\right)^2 \underset{k\to +\infty}{\longrightarrow}0.
\end{equation}
Moreover, by the equation of $\bar\eta$ we deduce
\begin{align*}
    \Vert \bar\eta''\Vert_{L^2} &\le \frac{\Vert \bar\eta(2\bar\theta +\bar\p^2-k)\Vert_{L^2}}{d}+ \frac{\Vert \bar\p^2\Vert_{L^2}}{d}\\
    &\le \frac{\Vert 2\bar\theta +\bar\p^2-k\Vert_{L^\infty}}{d}\Vert \bar\eta \Vert_{L^2} + \frac{\sqrt{\ell}}{d}\Vert \bar\p\Vert_{L^\infty} \\
    &\le \frac{2\sqrt{\ell}C}{d(C-2)} \Vert \bar\p\Vert_{L^\infty} + \frac{\sqrt{\ell}}{d} \Vert \bar\p\Vert_{L^\infty}\text{ by \eqref{2theta+phi^2-k_bounds}--\eqref{Eq:Dubeau2}}\\
    &\le \frac{\sqrt{\ell}}{d} \Vert \bar\p \Vert_{L^\infty} \frac{3C-2}{C-2},
\end{align*}
that is
\begin{equation}\label{Eq:eta''_L^2}
    \Vert \bar\eta''\Vert_{L^2} \le \frac{\sqrt{\ell}}{d} \Vert \bar\p \Vert_{L^\infty} \frac{3C-2}{C-2},
\end{equation}
whence, form Sobolev embedding, we finally obtain \eqref{Eq:eta'_L^infty}.

\end{proof}

 We are now in a position to prove Proposition \ref{Prop:Intermediate} (and, consequently, to conclude the proof of Theorem \ref{Th:Uniqueness}).
 
    \begin{proof}[Proof of Proposition \ref{Prop:Intermediate}]
    Note that for any $d,\nu,\ell>0$ and $C>3$ we have $k_2 \ge \frac{3+\sqrt{5}}{2}$, whence by Lemma \ref{Prop:1-bareta}, $1-\bar\eta>0$ in $(0;\ell)$. Let us also observe that
        \[ \frac{\abs{\left(\frac{\bar\theta}{1-\bar\eta}\right)'}^2}{\left(\frac{\bar\theta}{1-\bar\eta}\right)^3}= \left(\frac{\bar\theta'}{\bar\theta}-\frac{(1-\bar\eta)'}{1-\bar\eta}\right)^2 \frac{1-\bar\eta}{\bar\theta}= \frac{\bar\theta'^2}{\bar\theta^3}(1-\bar\eta) -2 \frac{\bar\theta' (1-\bar\eta)'}{\bar\theta^2} +\frac{(1-\bar\eta)'^2}{\bar\theta(1-\bar\eta)}.\]
        Moreover, by \eqref{Eq:Ide0}--\eqref{condition_C_K} we have 
        \begin{equation} \label{bound_1/theta}
            \frac{1}{\bar\theta}\le \frac{1}{\inf \bar\theta}\le \frac{C}{k(C-2)}.
        \end{equation} 
        Let $x_+$ be a maximum point of $\bar\eta$, by \eqref{condition_C_K} we also have
        \[ \bar\eta(x_+) k \frac{C-2}{C}\le \bar\eta(x_+) (2\bar\theta+\bar\p^2-k) \le \Vert\bar\p\Vert^2_{L^\infty},\]
        whence \[\bar\eta (x_+) \le \frac{C}{k(C-2)} \Vert\bar\p\Vert^2_{L^\infty}{\leq \frac{1}{C-2}}.\]
        In particular, we obtain that 
        \[ 0<{\frac{C-3}{C-2}= 1-  \frac{1}{(C-2)} }\le 1-\bar\eta<1\]
        for any $C>3$, whence
        \begin{equation}\label{bound_1/1-eta}
            \frac{1}{1-\bar\eta}\le {\frac{(C-2)}{(C-3)}}.
        \end{equation}
        By the Young inequality and Lemmata \ref{Prop:BHL_bartheta}--\ref{Prop:1-bareta}, we deduce that for any $\sigma>0$
        \begin{align*}
            \frac{\abs{\left(\frac{\bar\theta}{1-\bar\eta}\right)'}^2}{\left(\frac{\bar\theta}{1-\bar\eta}\right)^3}&=\frac{\bar\theta'^2}{\bar\theta^3}(1-\bar\eta) -2 \frac{\bar\theta' (1-\bar\eta)'}{\bar\theta^2} +\frac{(1-\bar\eta)'^2}{\bar\theta(1-\bar\eta)}\\
            &\le (1+\sigma)\frac{\bar\theta'^2}{\bar\theta^3}(1-\bar\eta) + \left( 1+\frac{1}{\sigma}\right) \frac{(1-\bar\eta)'^2}{\bar\theta(1-\bar\eta)}\\
            &< \frac{2}{3}(1+\sigma)+ \left( 1+\frac{1}{\sigma}\right)\Vert \bar\eta' \Vert_{L^\infty}^2 \frac{C}{k(C-2)} {\frac{(C-2)}{C-3}} \text{ by \eqref{bound_1/theta}--\eqref{bound_1/1-eta}},\\
        \end{align*}
        whence
        \begin{equation}\label{ineqtodifferentiate_sigma}
            \frac{\abs{\left(\frac{\bar\theta}{1-\bar\eta}\right)'}^2}{\left(\frac{\bar\theta}{1-\bar\eta}\right)^3} < \frac{2}{3} \left[ 1+\sigma +\frac{3}{2}\left( 1+\frac{1}{\sigma}\right)  \frac{C \Vert \bar\eta' \Vert_{L^\infty}^2}{{k(C-3)}}\right].
        \end{equation} 
        A one-dimensional study shows that the right-hand side is minimised by 
        \[ \sigma_{\text{opt}} = \sqrt{\frac{3}{2}\frac{C}{{k(C-3)}}} \Vert \bar\eta' \Vert_{L^\infty}.\]
        Replacing the optimal $\sigma_{\text{opt}}$ in \eqref{ineqtodifferentiate_sigma}, we obtain
        \begingroup\allowdisplaybreaks\begin{align*}
             \frac{\abs{\left(\frac{\bar\theta}{1-\bar\eta}\right)'}^2}{\left(\frac{\bar\theta}{1-\bar\eta}\right)^3}&< \frac{2}{3} \left[ 1+\sigma +\frac{3}{2}\left( 1+\frac{1}{\sigma}\right) \frac{C\Vert \bar\eta' \Vert_{L^\infty}^2}{{k(C-3)}}\right]\\
             &\le \frac{2}{3}\left[ 1 + 2 \sqrt{\frac{3}{2}\frac{C}{(C-2)k -C \| \bar\p\|_{L^\infty}^2}} \Vert \bar\eta' \Vert_{L^\infty} + \frac{3}{2}\frac{C \Vert \bar\eta' \Vert_{L^\infty}^2}{{k(C-3)}}\right]\\
             &= \frac{2}{3} \left( 1+  \sqrt{\frac{3}{2}\frac{C}{{k(C-3)}}} \Vert \bar\eta' \Vert_{L^\infty}\right)^2.
        \end{align*}\endgroup
Now, observe that
        \[ \frac{2}{3} \left( 1+ \sqrt{\frac{3}{2}\frac{C}{{k(C-3)}}} \Vert \bar\eta' \Vert_{L^\infty}\right)^2 \le 4\]
        if, and only if,
        \[  k\ge \frac{3}{2}\Vert \bar\eta'\Vert_{L^\infty} \frac{C}{(\sqrt{6}-1)^2 (C-2)} + \frac{C}{C-2}\Vert \bar\p\Vert^2_{L^\infty}.\]
        Hence, by \eqref{assumptionfor_theta/1-eta_XiBaiLi} and Lemma \ref{Prop:1-bareta}, we can conclude that 
        \[  \frac{\abs{\left(\frac{\bar\theta}{1-\bar\eta}\right)'}^2}{\left(\frac{\bar\theta}{1-\bar\eta}\right)^3}<4.\]
Theorem \ref{Th:generaluniqueness} now applies.    \end{proof}

\bibliographystyle{abbrv}
\bibliography{biblio}

\appendix
\section{Derivation of \eqref{Eq:MFC}}\label{Ap:Derivation}
The goal of this section is to derive \eqref{Eq:MFC}. As we mentioned earlier, this derivation is formal and relies on the existence of a maximiser $\alpha^*$ of the function
     \[ \alpha \longmapsto J(0,\alpha)=\iint_{(0;T)\times \O} \left( \theta_\alpha m_\alpha - \frac{1}{2}\alpha^2 m_\alpha \right) \, dx\, dt
     \]
     where we recall that 
     \[\partial_tm_\alpha-\nu\Delta m_\alpha+\n\cdot(\alpha m_\alpha)=0,\, m_\alpha(t=0,\cdot)=m_0\] and 
     \[ \partial_t \theta_\alpha-\Delta\theta_\alpha=\theta_\alpha(k-m_\alpha-\theta_\alpha),\, \theta_\alpha(t=0,\cdot)=\theta_0.\] 
Here, $m_0, \, \theta_0 \in L^2(\Omega)$, $m_0$ is a probability density and $\theta_0\geq 0,\, \not\equiv 0$. Observe that we are restricting ourselves to the case of feedback controls $\alpha$ depending only on $(t,x)$, which is not restrictive in this setting.
        
   We now fix $\alpha^*$, a maximiser of $J$. In order to characterise it, 
  assume the Gateaux differentiability of the map $\alpha\mapsto m_\alpha,\, \theta_\alpha, J(\alpha)$ (these differentiability results could be established rigorously using standard techniques from \cite{zbMATH06250859} for the continuity equation and from \cite{DFLLYQ} for the reaction-diffusion equation but this would make the presentation of this formal result more tedious). Namely, we assume that, for any admissible perturbation $(\delta \alpha)$ at a given admissible $\alpha$, the limits
  \[ \dot m:=\lim_{\e\to 0}\frac{m_{\alpha+\e(\delta\alpha)}-m_\alpha}\e \text{ in  a weak $L^2(0,T;W^{-1,2}(\O))$ sense}\] and 
  \[ \dot\theta:=\lim_{\e\to 0}\frac{\theta_{\alpha+\e(\delta\alpha)}-\theta_\alpha}\e \text{ in a strong $L^2(0,T;W^{1,2}(\O))$ sense} \] exist. In this case, it is clear that $(\dot m,\dot \theta)$ solve
  \begin{equation}\label{Eq:Dotm}
  \begin{cases}
  \partial_t \dot m-\nu\Delta \dot m+\n\cdot(\alpha \dot m)=-\n\cdot\left(\left(\delta \alpha\right)m\right)&\text{ in }(0;T)\times \O,\, 
  \\ \partial_\nu \dot m=0&\text{ on }(0;T)\times \partial \O,\, 
  \\ \dot m(t=0,\cdot)=0&\text{ in }\O
  \end{cases}
  \end{equation}
  and 
  \begin{equation}\label{Eq:DotTheta_Appendix}
  \begin{cases}
  \partial_t\dot\theta-d\Delta\dot\theta=\dot\theta\left(k-m-2\theta_\alpha\right)-\dot m \theta_\alpha &\text{ in }(0;T)\times \O,\, 
  \\ \partial_\nu \dot \theta=0&\text{ on }(0;T)\times \partial \O,\, 
  \\ \dot \theta(t=0,\cdot)=0&\text{ in }\O
  \end{cases}
  \end{equation}
  In particular, the Gateaux derivative of $J$ at $\alpha$ in the direction $(\delta\alpha)$ writes
  \[ \dot J(\alpha)[\delta\alpha]=\iint_{(0;T)\times \O} \left(\dot\theta m_\alpha+\theta_\alpha\dot m-\frac12|\alpha|^2\dot m-\langle \alpha,\delta \alpha\rangle  m_\alpha\right).\]
  In order to write $\dot J$ in a tractable form, we introduce the two adjoint states $(u_\alpha,\, \eta_\alpha)$ as the unique solutions of the backwards equations
 \begin{equation}\label{Eq:Dual}
        \begin{cases}
            -\partial_t \eta_\alpha-d \Delta \eta_\alpha - \eta_\alpha \left(k-m-2\theta\right)=m_\alpha& \text{in } \ (0,T)\times \Omega,\\
            -\partial_t u_\alpha -\nu \Delta u_\alpha- \alpha \cdot \nabla u_\alpha= \theta_\alpha \left(1-\eta_\alpha\right)-\frac{\abs{\alpha}^2}{2} \quad& \text{in } \ (0,T)\times \Omega,\\
            \partial_\nu \eta_\alpha = 0, \,\partial_\nu u_\alpha =0 &\text{ on } \ (0;T)\times \partial\Omega,\\
            \eta_\alpha(T,\cdot)= u_\alpha(T,\cdot)=0 \quad &\text{in } \ \Omega,
        \end{cases}
    \end{equation}
Multiplying \eqref{Eq:DotTheta_Appendix} by $\eta_\alpha$ and integrating by parts, multiplying \eqref{Eq:Dotm} by $u_\alpha$ and integrating by parts, we obtain 
\[\iint_{(0;T)\times \O} \dot\theta m_\alpha=-\iint_{(0;T)\times \O} \theta_\alpha\eta_\alpha\dot m\] and 
\[ \iint_{(0;T)\times \O} \left(\theta_\alpha(1-\eta_\alpha) \dot m-\frac{|\alpha|^2}2\dot m\right)=\iint_{(0;T)\times \O} \langle \delta \alpha,\n u_\alpha\rangle m_\alpha,\] leading to 
\[ \dot J(\alpha)[\delta\alpha]=\iint_{(0;T)\times \O} \left\langle \delta\alpha,\n u_\alpha-\alpha \right\rangle m_\alpha . 
\]
In particular, we deduce that if $\alpha^*$ is optimal then 
\[ \n u_{\alpha^*}=\alpha^*,  \] whence the equation on $u_\alpha$ rewrites as the Hamilton-Jacobi equation 
\[-\partial_tu_{\alpha^*}-\nu\Delta u_{\alpha^*}-\frac12|\n u_{\alpha^*}|^2=\theta_{\alpha^*}\left(1-\eta_{\alpha^*}\right),\] leading to the desired \eqref{Eq:MFC}.
  
\end{document}